\documentclass[11pt]{amsart}
\usepackage{amssymb,amsmath,latexsym,amsthm}
\usepackage{amsfonts}
\usepackage{enumerate}
\usepackage[abbrev,lite]{amsrefs}
\usepackage{enumitem}
\usepackage[margin=1.3in]{geometry}
\usepackage{color}

\usepackage{dsfont}
\usepackage{mathrsfs}

\usepackage{hyperref}
\usepackage{xcolor}
\hypersetup{
    linktocpage,
    colorlinks, 
    linkcolor={red!80!black},
    citecolor={blue!80!black}, 
    urlcolor={blue}
}

\newtheorem{proposition}{Proposition}[section]
\newtheorem{theorem}[proposition]{Theorem}
\newtheorem{corollary}[proposition]{Corollary}
\newtheorem{lemma}[proposition]{Lemma}
\theoremstyle{definition}
\newtheorem{definition}[proposition]{Definition}
\newtheorem{remark}[proposition]{Remark}

\numberwithin{equation}{section}

\newcommand{\dive}{\mathrm{div}}
\newcommand{\dd}{\mathrm{d}}
\newcommand{\e}{\mathrm{e}}
\newcommand{\R}{\mathbb{R}}
\newcommand{\V}{\mathbf{V}}
\def\H{\mathbf{H}}
\newcommand{\Q}{\mathcal{Q}}
\newcommand{\A}{\mathcal{A}}
\newcommand{\M}{\mathcal{M}}
\newcommand{\la}{\langle}
\newcommand{\ra}{\rangle}
\newcommand{\de}{\partial}
\newcommand{\eps}{\varepsilon}
\newcommand{\vv}{\boldsymbol{v}}
\newcommand{\uu}{\boldsymbol{u}}
\newcommand{\kk}{\boldsymbol{k}}
\newcommand{\tF}{\widetilde{F}}
\newcommand{\T}{t_1}

\newcommand{\hv}{\hat{\boldsymbol{v}}}
\newcommand{\he}{\hat{e}}
\newcommand{\hq}{\hat{q}}
\newcommand{\hT}{\widehat{T}}
\newcommand{\ho}{\hat{\omega}}

\def\bs{\boldsymbol}
\def\ddt{\frac{\mathrm{d}}{\mathrm{d}t}}

\newcommand{\abs}[1]{\lvert #1 \rvert }
\newcommand{\norm}[1]{\| #1 \| }

\title[Primitive equations of the atmosphere in presence of vapor saturation]
{The primitive equations of the atmosphere in presence of vapor saturation}

\author[M. Coti Zelati, A. Huang, I. Kukavica, R. Temam, M. Ziane]
{Michele Coti Zelati, Aimin Huang, Igor Kukavica, \\ Roger Temam, and Mohammed Ziane}

\date{\today}

\address{Indiana University - Institute for Scientific Computing and Applied Mathematics 
\newline\indent
Rawles Hall, Bloomington, IN 47405, USA}
\email{micotize@indiana.edu {\rm (M.\ Coti Zelati)}}
\email{aimhuang@indiana.edu {\rm (A.\ Huang)}}
\email{temam@indiana.edu {\rm (R.\ Temam)}}

\address{University of Southern California - Department of Mathematics 
\newline\indent
3620 S. Vermont Ave., KAP 108, Los Angeles, California 90089, USA}
\email{kukavica@usc.edu {\rm (I.\ Kukavica)}}
\email{ziane@math.usc.edu {\rm (M.\ Ziane)}}



\begin{document}

\begin{abstract}
A modification of the classical primitive equations of the atmosphere is considered
in order to take into account important phase transition phenomena due to air saturation and condensation. We 
provide a mathematical formulation of the problem that appears to be new in this setting, by making use of
differential inclusions and variational inequalities, and which allows to develop a rather complete theory for the
solutions to what turns out to be a nonlinearly coupled system of non-smooth partial differential equations. Specifically we prove global
existence of quasi-strong and strong solutions, along with uniqueness results and maximum principles of physical interest.
\end{abstract}

\maketitle

\tableofcontents

\section{Introduction}
\noindent The primitive equations (PEs) represent the classical model for the study of climate
and weather prediction,  describing the motion of the atmosphere when the hydrostatic
assumption is enforced \cites{G,Hal71,HW80, TM05, Ric88}. To the best of our knowledge, their mathematical 
study was initiated by Lions, Temam and Wang in \cites{LTW92a, LTW92b, LTW93}. This research field has 
quickly developed, now attracting a large number of researchers over the last two decades. 

According to  the classical theory of atmospheric dynamics \cites{G,Hal71,HW80}, for dry adiabatic
motions a complete system of equations consists of the vector equation of motion
\begin{equation}\label{eq:con1}
\rho \frac{\dd \mathbf{V_3}}{\dd t}+\rho \Omega \times \mathbf{V_3}+\nabla_3 p+\rho\mathbf{g}=D,
\end{equation}
the equation of state
\begin{equation}\label{eq:con2}
p=R\rho T,
\end{equation}
the equation of continuity
\begin{equation}\label{eq:con3}
\frac{\dd \rho}{\dd t}+\rho\dive_3 \mathbf{V_3}=0,
\end{equation}
and the first law of thermodynamics
\begin{equation}\label{eq:con4}
\frac{\dd T}{\dd t}-\frac{RT}{c_p p}\frac{\dd p}{\dd t}=Q_T.
\end{equation}
Above, the following quantities play an important role:
\begin{itemize}
	\item $\mathbf{V_3}=(\vv,w)=$ velocity of the wind, where $\vv=(u,v)$ is the horizontal velocity;
	\item $\rho=$ density, $p=$ pressure, $T=$ temperature;
	\item $\mathbf{g}=(0,0,-g)=$ the gravity, $0<c_p=$ specific heat, $0<R=$ specific gas constant;
	\item $\Omega=$ angular velocity of the earth;
	\item $D,Q_T=$ dissipation terms.
\end{itemize}
In their general three dimensional form, these equations are too complicated to be treatable
both from the theoretical and the computational side. The most common physical simplification is due
to the observation that the vertical dimension is usually much smaller than the horizontal
one. 
Based
on the hydrostatic approximation \cite{AG},
this leads to the derivation of the \emph{primitive equations}
\cites{G,Ped87,LTW92a,LTW92b,LTW93,PTZ08}.
The hydrostatic assumption introduces the equation
\begin{equation}\label{eq:hydro}
\frac{\de p}{\de z}=-\rho g,
\end{equation}
which corresponds to the simplified form of the equation of conservation of 
momentum in the vertical direction. In its turn, \eqref{eq:hydro} shows that $p$
is a decreasing function of $z$, which allows the use of $p$ as the vertical coordinate. In the $(x,y,p)$ system, 
the equation \eqref{eq:con3} now becomes
\begin{equation}\label{eq:divfree}
\frac{\de u}{\de x}+\frac{\de v}{\de y}+\frac{\de \omega}{\de p}=0,
\end{equation}
where $\omega$ (different from $w$) is now the vertical velocity of the wind in the $(x,y,p)$ system, defined by
\begin{equation}\label{eq:ommm}
\omega=\frac{\dd p}{\dd t},
\end{equation}
where the total derivative in the $(x,y,p)$-coordinates reads
\begin{equation}\label{eq1.e1}
\frac{\dd}{\dd t}= \frac{\de}{\de t}+\vv\cdot \nabla +\omega \frac{\de}{\de p}.
\end{equation}
The mathematical 
form of the equations \eqref{eq:con1}, \eqref{eq:con4}, and \eqref{eq:divfree} in the $(x,y,p)$ system makes these
equations very similar to the Navier-Stokes equations of incompressible fluids \cites{S01,T95,T01} and make their mathematical
theory feasible; see e.g.~\cite{LTW92a} and the review article \cite{PTZ08} which contains a large list
of mathematical references.

When moisture is included, an equation for the conservation of water must be added, which
is the case in e.g.~\cites{GH1,GH2,LTW92a,PTZ08}. However, in these works, the equation of conservation
of moisture, corresponding to the variable $q$, $0\leq q\leq 1$, is simply an equation of transport which does not account 
for the changes of phase, concentration/evaporation, and rain. In this work we will mainly be 
concerned with
modifications coming from this change of model, with the aim of providing a rigorous mathematical 
framework
for the study of such systems of equations. In earlier works \cites{CFTT, CT12}, two of the authors have studied the equations of the
humid atmosphere with saturation by making the simplifying assumption that the velocity
of the air $\uu=(\vv,\omega)$ is given. In this article we address the whole problem, by coupling 
the methods developed in \cites{CFTT, CT12} with the tools developed to study the
three dimensional primitive equations \cite{CT06, Kob1, Kob2, KPRZ, KZ07,KZ08}.

\subsection{The introduction of moisture}

The equation of water vapor may be obtained in a similar manner to the derivation of the continuity
equation. Calling $q$ the specific humidity, namely, the mass of water vapor in a unit mass of moist air,
we can write its conservation equation \cites{Hal71,HW80} as
\begin{equation}\label{eq:qqqq}
\frac{\dd q}{\dd t}= \frac{S}{\rho}+D_q,
\end{equation}
where $D_q$ is a suitable form of dissipation accounting for horizontal and vertical diffusion and 
$S$ are sources or sinks of water vapor per unit volume per unit time. This equation extends the 
classical conservation equation from e.g.~\cites{HW80,Hal71} by adding the dissipation term $D_q$.
In general, a sink of water vapor arises from condensation or evaporation from saturated air, in which case
\begin{equation}\label{eq:sssss}
\frac{S}{\rho}=\frac{\dd q_s}{\dd t},
\end{equation}
where $q_s$ is the saturation humidity, otherwise called the saturation concentration. In general,
$q_s$ will be either constant or a nonlinear function of the temperature $T$, and will satisfy 
\begin{equation}\label{eq:qssss}
\frac{\dd q_s}{\dd t}= \frac{\omega}{p}F(T),
\end{equation}
for some function $F$ whose expression will be discussed in detail in Section \ref{sub:nonlinsec}.
The common coordinate system to study the primitive equations of the atmosphere 
is, as we said,  the $(x,y,p)$ coordinate system, where the pressure $p$ replaces the vertical
coordinate $z$. 
Expanding equation \eqref{eq:qqqq}, in view of \eqref{eq:sssss}-\eqref{eq:qssss} 
and \eqref{eq:ommm}, we obtain that $q$ undergoes the evolution equation
\begin{equation}\label{eq:qqqq2}
\frac{\de q}{\de t}+ \vv \cdot \nabla q + \omega \frac{\de q}{\de p}= \delta\, \frac{\omega}{p}\, F(T) + D_q,
\end{equation}
where $\delta$ is defined as
$$
\delta=\begin{cases}
1,\quad \text{if } \omega<0 \text{ and } q> q_s,\\
0,\quad \text{otherwise}.
\end{cases}
$$
Roughly speaking, the contribution by $F(T)$ is assumed to apply only during condensation ($q> q_s$) and
requires a negative $\omega$ (upward motion). Viewing $q_s$ as a threshold, this precisely describes the change of phase
which the specific humidity $q$ obeys. The classical terminology is the following:
\begin{itemize}
\item $q<q_s$: under-saturated regime;
\item $q=q_s$: saturation/condensation;
\item $q>q_s$: over-saturated regime.
\end{itemize}
From the point of view of partial differential equations, \eqref{eq:qqqq2} introduces the mathematical 
difficulty of dealing with a nonlinear and discontinuous right hand side, making more challenging the proof of suitable
well-posedness results. 

\subsection{A modified law of thermodynamics}
The expression for $\dd q_s/\dd t$ in  \eqref{eq:qssss} may also be used in the thermodynamic equation
$$
-L \frac{\dd q_s}{\dd t}=c_p \frac{\dd T}{\dd t}- \frac{RT}{p}\frac{\dd p}{\dd t},
$$
in order to properly modify the conservation equation \eqref{eq:con4}. We then have
\begin{equation}\label{eq:newthermo}
\frac{\dd T}{\dd t}-\frac{RT}{c_pp}\frac{\dd p}{\dd t}=-\delta L \frac{\omega}{c_pp} F(T)+ D_T,
\end{equation}
where $\delta$ is the same as above. The first term in the right hand side of \eqref{eq:newthermo} corresponds
to the energy effect of condensation or evaporation.
Here, $D_T$ is another form of dissipation accounting, e.g.~for
conduction or turbulence. Written in the $(x,y,p)$ coordinate system, the above equation reads
\begin{equation}\label{eq:qqqq3}
\frac{\de T}{\de t}+ \vv \cdot \nabla T + \omega \frac{\de T}{\de p}-\frac{RT}{c_pp}\omega=-\delta L \frac{\omega}{c_pp} F(T)+ D_T.
\end{equation}

\subsection{Structure of the article}
In the next section we rephrase, following \cites{CFTT,CT12},  the system as a set of differential inclusions, for which we provide suitable definitions 
of quasi-strong and strong solutions in Section \ref{sec:mathset}. Section \ref{sec:approx111}  is devoted to the construction 
of a regularized approximated problem which is useful to prove the local-in-time existence of quasi-strong solutions in Section \ref{sec:localquasi}.
Local and global existence of strong solutions is proved in Section \ref{sec5}. The question of uniqueness of quasi-strong and 
strong solutions is addressed in Section \ref{sec:uniq}, and it is very much related to the maximum principle results
derived in Section \ref{sec:maxprinc}. Finally, we conclude the article with the Appendix \ref{sec-kz-idea}, in which a modification
of a result of \cites{KZ07,KZ08} is considered.

\section{The primitive equations as differential inclusions}

\noindent We start by defining our problem in a physical bounded domain, imposing boundary
conditions  and specifying the precise form of the dissipation terms appearing in \eqref{eq:con1}, \eqref{eq:qqqq2} and
\eqref{eq:qqqq3}. The equations for temperature and specific humidity are rephrased as differential inclusions,
in order to make the discontinuities in the right-hand-sides of  \eqref{eq:qqqq2} and
\eqref{eq:qqqq3} treatable from the mathematical viewpoint.

\subsection{The basic model}
Let $\M$ be a cylinder in $\mathbb{R}^3$ of the form
$$
\M = \{(x,y,p)\,:\,(x,y)\in \M',\, p\in (p_0,p_1)\},
$$
where $\M'$ is a smooth bounded domain in $\mathbb{R}^2$, and $p_0<p_1$ are positive constants.
We denote by $\nabla, \Delta$ and $\dive$ the two-dimensional gradient, Laplacian and divergence operators, respectively, that is
$$
\nabla = (\de_x,\de_y),\quad \Delta= \de_x^2 + \de_y^2,\quad \dive=\de_x + \de_y.
$$
Analogously, the symbols $\nabla_3,\Delta_3$  and $\dive_3$  indicate their three-dimensional versions.
The (viscous) primitive equations of the atmosphere read
\begin{align}
\label{eq1.1.1}&\frac{\de \vv}{\de t} + (\vv\cdot\nabla) \vv + 
\omega \frac{\de \vv}{\de p} + f \kk\times \vv + \nabla \Phi  + \A_{\vv} \vv = \boldsymbol S_{\vv},\\
\label{eq1.1.2}&\frac{\de\Phi}{\de p} + \frac{RT}{p}=0,\\
\label{eq1.1.3}&\text{div}\, \vv + \frac{\de\omega}{\de p}=0,\\
\label{eq1.1.4}&\frac{\de T}{\de t} + \vv\cdot\nabla T + \omega \frac{\de T}{\de p} - \frac{R T}{c_p p}\omega +\A_T T =S_T,\\
\label{eq1.1.5}&\frac{\de q}{\de t} + \vv\cdot\nabla q + \omega \frac{\de q}{\de p} + \A_q q = S_q,\\
\label{eq1.1.6}&p=R\rho T.
\end{align}
Here, $\uu=(\vv,\omega)$ is the three-dimensional velocity vector, $\rho, p, T$ 
are the density, pressure and the temperature distribution, and $q$ is the specific humidity, measuring 
the amount of vapor in the air. 
In \eqref{eq1.1.1}, $f$ is the Coriolis force parameter and $\kk$ 
is the unit vector in the direction of the poles (from south to north). As the equations we consider here are the viscous PEs of the atmosphere, the symbols $\A_{\vv}$, $\A_T$ and $\A_q$ 
denote diffusion operators, with suitable eddy viscosity coefficients:
$$
\A_{\star} = - \mu_{\star}\Delta  - \nu_{\star}\frac{\de}{\de p}\left( \left(\frac{gp}{R\overline T}\right)^2 \frac{\de }{\de p}\right),
$$
where $\star$ can either be $\vv$, $T$, or $q$, and $\overline T=\overline T(p)$ is a given average 
temperature over the isobar with pressure $p$, for which we assume the existence
of two positive constants $\overline T_*$ and $\overline T^*$ such that
\begin{equation}\label{eq:tstar1}
\overline T_*\leq \overline T(p)\leq \overline T^*.
\end{equation}
Concerning the right hand sides, $S_T$ corresponds to the sum of the heating of the sun and the heat added or removed by condensation or evaporation; $S_q$ represents the amount of water added or removed by condensation or evaporation. Finally, $\boldsymbol S_{\vv}$, which vanishes in reality, is a forcing term usually added for mathematical generality  and to possibly handle nonhomogenous boundary conditions.

\subsection{The primitive equations with saturation}

When studying the climate dynamics around the equator, the humidity equation,
describing the ratio of vapor in the air, becomes very important and it is necessary 
to account for the possible saturation of vapor leading to condensation (clouds) and 
rain. In the recent works \cites{CT12,CFTT}, the authors considered 
the coupling of the humidity equation and the temperature equation 
with a \emph{given} velocity vector field $\uu=(\vv,\omega)$. These two differential inclusions replacing \eqref{eq1.1.4} and \eqref{eq1.1.5} read (see \cites{Hal71,HW80}):
\begin{align}
\label{eq1.1.7}&\frac{\de T}{\de t} + \vv\cdot\nabla T + \omega\frac{\de T}{\de p}-\frac{R}{c_p p}\omega T+ 
\A_T T  \in \frac{L}{c_p p}\omega^{-}H(q-q_s)F(T) +S_T,\\
\label{eq1.1.8}&\frac{\de q}{\de t} + \vv\cdot\nabla q + \omega \frac{\de q}{\de p} + \A_q q \in -\frac{1}{p}\omega^{-}H(q-q_s)F(T)+S_q.
\end{align}
Here $\omega^{-}=\max\{-\omega,0\}$ is the negative part of $\omega$, 
and $H(q-q_s)$ is the Heaviside multivalued function, i.e.,
\begin{equation}
H(r)=\begin{cases}
0, &r<0,\\
[0,1],\hspace{6pt}&r=0,\\
1,&r>0.
\end{cases}
\end{equation}
The papers \cites{CT12,CFTT} provide the existence, 
uniqueness, and maximum principles of  weak solutions 
to the equations \eqref{eq1.1.7}--\eqref{eq1.1.8}, with the velocity
 $\uu$ given in some suitable Sobolev spaces.

In this article, we consider the full nonlinear PEs in the presence of vapor saturation.
Specifically, we replace  equations \eqref{eq1.1.4}--\eqref{eq1.1.5} 
with the differential inclusions  \eqref{eq1.1.7}--\eqref{eq1.1.8} in order to take the saturation phenomenon into account.
More precisely, the full system under study reads now
\begin{align}
&\label{eq1.1.1b}\frac{\de \vv}{\de t} + (\vv\cdot\nabla)\vv + \omega \frac{\de \vv}{\de p} + f \kk\times \vv + \nabla \Phi  + \A_{\vv} \vv = \boldsymbol S_{\vv},\\
&\label{eq1.1.2b}\frac{\de\Phi}{\de p} + \frac{RT}{p}=0,\\
&\label{eq1.1.3b}\text{div}\, \vv + \frac{\de\omega}{\de p}=0,\\
&\label{eq1.1.4b}\frac{\de T}{\de t} + \vv\cdot\nabla T  + \omega\frac{\de T}{\de p}-\frac{R}{c_p p}\omega T+ \A_T T  \in \frac{L}{c_p p}\omega^{-}H(q-q_s)F(T)+S_T,\\
&\label{eq1.1.5b}\frac{\de q}{\de t} + \vv\cdot\nabla q + \omega \frac{\de q}{\de p} + \A_q q \in -\frac{1}{p}\omega^{-}H(q-q_s)F(T)+S_q,\\
&\label{eq1.1.6b}p=R\rho T.
\end{align}
Our aim here is to study the coupled system \eqref{eq1.1.1b}--\eqref{eq1.1.6b} and prove the existence and uniqueness of the global (quasi)-strong solutions defined in Sections \ref{def:quasi-strong} and \ref{def:strong}, for a \emph{constant} saturation concentration $q_s\in(0,1)$.

\subsection{Nonlinear terms}\label{sub:nonlinsec}
An important difference between our system and the one considered in the classical references 
\cites{HW80,Ped87} (see also \cites{LTW92a, LTW95}) are the nonlinear terms. Firstly, the temperature equation  \eqref{eq1.1.4b}
involves the nonlinear (and possibly anti-dissipative) term 
$$
-\frac{R}{c_p p}\omega T
$$
on the left hand side. This requires some care from the very beginning, as shown in \cite{ET01}.

Secondly, in \eqref{eq1.1.4b} and \eqref{eq1.1.5b}, the nonlinearity $F\colon\mathbb{R}\to\mathbb{R}$ 
is obtained by setting \cites{Hal71,RY89}
\begin{equation}\label{eq:nonlinF}
F(\xi)= q_s \xi\left(\frac{LR-c_pR_v\xi}{c_pR_v\xi^2+q_s L^2}\right),
\end{equation}
where $c_p$, $L$ and $R$ are the positive constant described above, and $R_v$ is equal to the gas constant for water vapor. 
By a direct calculation, we see that $F$ is a globally Lipschitz bounded function, namely
\begin{equation}\label{N1}
|F(\xi_1)-F(\xi_2)|\leq c_F|\xi_1-\xi_2|, \qquad \forall \xi_1,\xi_2\in \mathbb{R},
\end{equation}
and
\begin{equation}\label{N2}
|F(\xi)|\leq C_F, \qquad \forall \xi\in \mathbb{R}.
\end{equation}
Since $F(0)=0$, we also obtain from \eqref{N1} that
\begin{equation}\label{N3}
|F(\xi)|\leq c_F|\xi|, \qquad \forall \xi\in \mathbb{R}.
\end{equation}
Moreover,
\begin{equation}\label{N4}
F(\xi_0)=0\quad\text{for}\quad \xi_0=\frac{L R}{c_pR_v}.
\end{equation}
Therefore,
\begin{equation}\label{eq:nonbound}
F(\xi)\geq 0 \quad \Leftrightarrow \quad \xi\in [0,\xi_0].
\end{equation}

\subsection{Boundary and initial conditions}
The boundary of $\M$ is partitioned into three parts as $\de\M = \Gamma_i\cup\Gamma_b\cup\Gamma_\ell$, where
\begin{equation}
\begin{aligned}
&\Gamma_i = \{(x,y,p)\in\overline\M\,:\, p=p_1\},\\
&\Gamma_u= \{(x,y,p)\in\overline\M\,:\, p=p_0\},\\
&\Gamma_\ell = \{(x,y,p)\in\overline\M\,:\, (x,y)\in \de \M',\,p_0\leq p\leq p_1\}.
\end{aligned}
\end{equation}
We supplement system \eqref{eq1.1.1b}--\eqref{eq1.1.6b}
with the following physically relevant (homogeneous) boundary conditions: 
wind-driven on the top surface and free-slip and non-heat flux on the 
side walls and the bottom (see e.g.~\cites{CT06,CT12}):
\begin{alignat}{4}
\label{eq1.7}\text{on }\Gamma_i:\quad& \frac{\de \vv}{\de p} = -\frac{\alpha_{\vv}}{\nu_{\vv}}\vv,\quad & &\omega=0, 
\quad & & \frac{\de T}{\de p} = -\frac{\alpha_T}{\nu_T}T,\qquad & &\frac{\de q}{\de p}=-\frac{\alpha_q}{\nu_q} q;\\
\label{eq1.8}\text{on }\Gamma_u:\quad &\frac{\de \vv}{\de p} = 0,\quad & &\omega=0,\quad & &\frac{\de T}{\de p} =0,\quad & &\frac{\de q}{\de p}=0;\\
\label{eq1.9}\text{on }\Gamma_\ell:\quad &\vv\cdot \boldsymbol n=0,\qquad & &\frac{\de(\vv\cdot\boldsymbol\tau)}{\de \boldsymbol n}=0,\qquad & &\frac{\de T}{\de \boldsymbol n}=0,\quad & &\frac{\de q}{\de \boldsymbol n} =0,
\end{alignat}
where $\boldsymbol n$ and $\boldsymbol \tau$ are the unit normal and tangent vectors to $\Gamma_\ell$ respectively and $\alpha_T,\alpha_q>0$ are given positive constants.
In addition, we supplement system \eqref{eq1.1.1b}--\eqref{eq1.1.6b} with the initial conditions
\begin{equation}\label{eq1.10}
\begin{aligned}
&\vv(x,y,p,0) = \vv_0(x,y,p),\\
&T(x,y,p,0) = T_0(x,y,p),\\
&q(x,y,p,0)=q_0(x,y,p).
\end{aligned}
\end{equation}

\begin{remark}
In limited area atmospheric models, the free-slip boundary conditions for the lateral boundary in \eqref{eq1.9} are more appropriate to avoid artificial boundary layer (see 
e.g.~\cite{TS82}). We could also consider the no-slip boundary conditions or periodic boundary conditions for the lateral boundary and the main results in this article still hold for these boundary conditions.
\end{remark}

\begin{remark}
A relevant non-homogeneous version of the above boundary conditions \eqref{eq1.7}--\eqref{eq1.9} can be be written as follows:
\begin{alignat}{4}
\label{eq1.7b}\text{on }\Gamma_i:\ & \frac{\de \vv}{\de p} = \frac{\alpha_{\vv}}{\nu_{\vv}}((p_1-p_0)\vv_*-\vv),\quad & &\omega=0, 
 & & \frac{\de T}{\de p} = \frac{\alpha_T}{\nu_T}(T_*-T),\quad & &\frac{\de q}{\de p}=\frac{\alpha_q}{\nu_q}(q_*-q);\\
\label{eq1.8b}\text{on }\Gamma_u:\ &\frac{\de \vv}{\de p} = 0, & &\omega=0,\quad& &\frac{\de T}{\de p} =0,\quad & &\frac{\de q}{\de p}=0;\\
\label{eq1.9b}\text{on }\Gamma_\ell:\ &\vv\cdot \boldsymbol n=0,  \qquad\qquad\frac{\de(\vv\cdot\boldsymbol\tau)}{\de \boldsymbol n}=0, & &\quad & &\frac{\de T}{\de \boldsymbol n}=0,\quad & &\frac{\de q}{\de \boldsymbol n} =0.
\end{alignat}
In this setting, $\vv_*(x,y)$ is the wind stress on the ocean surface and 
$T_*(x,y)$ and $q_*(x,y)$ are typical temperature and specific humidity 
distributions at the bottom surface of the atmosphere, respectively. Due to the boundary conditions \eqref{eq1.7b}--\eqref{eq1.9b}, it is natural to assume that $\vv_*$, $T_*$ and $q_*$ satisfy the boundary compatibility conditions
\begin{equation}\label{eq1.8.1}
\begin{split}
	\vv_*\cdot \boldsymbol n&=0,\quad \frac{\de (\vv_*\cdot\boldsymbol\tau)}{\de \boldsymbol n}=0,\qquad \text{ on }\de\M',\\
	\frac{\de T_*}{\de \boldsymbol n}&=0,\quad \frac{\de q_*}{\de \boldsymbol n}=0,\qquad\quad \text{ on }\de\M'.
\end{split}
\end{equation}
As observed in \cite[Remark~1]{CT06}, if we make the following variables change
\begin{equation}
	\begin{cases}
		\tilde{\boldsymbol v} =\boldsymbol v- (p_1-p_0)\vv_*,\\
		\widetilde T = T - T_*,\\
		\tilde q = q- q_*,
	\end{cases}
\end{equation}
then $(\tilde{\boldsymbol v},\widetilde T,\tilde q)$ satisfies the homogeneous boundary conditions \eqref{eq1.7}--\eqref{eq1.9} thanks to the compatibility boundary conditions \eqref{eq1.8.1}, and the extra terms involving $(\vv_*,T_*,q_*)$ appearing in the new set of equations similar to those \eqref{eq1.1.1b}--\eqref{eq1.1.6b} are lower order terms and easy to handle. Hence for simplicity and without loss of generality we assume that $\vv_* =T_* = q_*=0$, 
corresponding to our choice \eqref{eq1.7}--\eqref{eq1.9}. Therefore, the results presented here are still valid for general $(\vv_*,T_*,q_*)$ provided these quantities  are smooth enough.
\end{remark}

\section{Mathematical setting}\label{sec:mathset}
\noindent The weak formulation of the system \eqref{eq1.1.1b}-\eqref{eq1.1.6b} along with its boundary condition
requires the introduction of a rather large set of functional analytic tools, including a variational inequality
to represent the Heaviside graph as the subdifferential of a convex functional. This section is therefore devoted to
making more precise the mathematical formulation of the equations under study.

\subsection{The potential temperature}\label{sub:pottemp}
In order to eliminate the demanding nonlinear term 
$$
\frac{R}{p}\omega T
$$ 
in the 
temperature equation \eqref{eq1.1.4b}, we introduce the 
so-called potential temperature
\begin{equation}\label{eq2.1.1}
\theta=T\left(\frac{p_0}{p}\right)^{R/c_p}.
\end{equation}
Using $\omega=\dd p/\dd t$ where $\dd/\dd t$ is defined in \eqref{eq1.e1}, a direct computation shows that
\begin{equation}\label{eq2.1.3}
\frac{\dd T}{\dd t} -\frac{R\omega}{c_pp}T=\left(\frac{p}{p_0}\right)^{R/c_p}\frac{\dd \theta}{\dd t}.
\end{equation}
In this way, the equation \eqref{eq1.1.4b} becomes
\begin{equation}\label{eq2.1.3b}
\frac{\de \theta}{\de t}+\vv \cdot \nabla \theta +\omega \frac{\de \theta}{\de p}+\A_\theta \theta
 \in \frac{L}{c_p p} \left(\frac{p_0}{p}\right)^{R/c_p}\omega^- H(q-q_s) \widetilde F (p,\theta) + S_\theta,
\end{equation}
where the operator $\A_T$ is replaced by
$$
\A_\theta=-\mu_T\Delta-\nu_T\left(\frac{p_0}{p}\right)^{R/c_p}\frac{\de}{\de p}\left(\frac{gp}{R\overline T}\right)^2\frac{\de}{\de p}\left(\frac{p}{p_0}\right)^{R/c_p},
$$
and the nonlinear term $F$ takes the form
\begin{align*}
\tF(p,\theta)= F\left[ \left(\frac{p}{p_0}\right)^{R/c_p}\theta\right] .
\end{align*}
Note that, thanks to the properties \eqref{N1}--\eqref{N3} of $F$ and the fact that $p\in[p_0,p_1]$, we have
\begin{equation}\label{eq:tF}
|\tF(p,\xi)|\leq C_{\tF},\qquad |\tF(p,\xi_1)-\tF(p,\xi_2)|\leq c_{\tF}|\xi_1-\xi_2|, \qquad  |\tF(p,\xi)|\leq c_{\tF}|\xi|,
\end{equation}
where the last inequality follows from the fact that $\tF(p,0)=F(0)=0$ for every $p$. 
In the same way, the boundary conditions become
\begin{equation}
\label{eq2.1.12} \left(\nu_T\frac{\de \theta}{\de p}+ \alpha_\theta \theta\right)\big|_{\Gamma_i} =0,\qquad  
\left(\frac{\de \theta}{\de p}+\frac{R}{c_pp_0}\theta\right)\big|_{\Gamma_u} =0,\qquad  \frac{\de \theta}{\de \boldsymbol n}\big|_{\Gamma_\ell}=0,
\end{equation}
where 
$$
\alpha_\theta=\alpha_T + \nu_T\frac{R}{c_pp_1},
$$
and the initial condition turns into
\begin{equation}
\label{eq2.1.14}\theta(x,y,p,0) = \theta_0(x,y,p)=T_0(x,y,p)\left(\frac{p_0}{p}\right)^{R/c_p}.
\end{equation}

\begin{remark}
It is clear that, since $0<p_0\leq p\leq p_1$, the properties that we will derive for $\theta$ will be translated into analogous properties for the temperature $T$.
\end{remark}

\subsection{New formulation}
Integrating \eqref{eq1.1.2b} in the $p$-direction gives
\begin{equation}\label{eq2.1.4}
\begin{split}
\Phi(x,y,p,t)&=\Phi_s(x,y,t) + \int_p^{p_1} \frac{R}{p'}T(x,y,p',t)\dd p'\\	
&=\Phi_s(x,y,t) + \int_p^{p_1} \frac{R}{p'}\left(\frac{p'}{p_0}\right)^{R/c_p}\theta(x,y,p',t)\dd p',
\end{split}
\end{equation}
where $\Phi_s=\Phi_s(x,y,t)$ is the pressure at the bottom of the atmosphere when $p=p_1$.
In the same manner, using the boundary conditions \eqref{eq1.8}, we infer from \eqref{eq1.1.3b} that $\omega=\omega(\vv)$ satisfies
\begin{equation}\label{eq2.1.5}
\omega(x,y,p,t) = \int_{p}^{p_1} \nabla\cdot \vv(x,y,p',t)\dd p',
\end{equation}
and the following constraint must be satisfied:
\begin{equation}\label{eq2.1.6}
\int_{p_0}^{p_1} \nabla\cdot \vv(x,y,p,t)\dd p =  \nabla\cdot\int_{p_0}^{p_1} \vv(x,y,p,t)\dd p =0.
\end{equation}
We aim to write the fully nonlinear PEs in the \emph{prognostic} variables, namely in
$(\vv, \theta, q)$. 
The other variables $T,\rho,\Phi,\omega$ can be determined by 
\eqref{eq1.1.6b}, \eqref{eq2.1.1}, \eqref{eq2.1.4}, and \eqref{eq2.1.5}, and  they 
are called the \emph{diagnostic} variables. Defining
\begin{equation}\label{eq:DDDDD}
D(\omega, \theta,q)=\frac{1}{p}\omega^{-}H(q-q_s)\tF (p,\theta),
\end{equation}
the PEs in the \emph{prognostic} variables $(\vv, \theta, q)$ read
\begin{align}
\label{eq2.1.7b}&\frac{\de \vv}{\de t} + \vv\cdot \nabla\vv + \omega \frac{\de \vv}{\de p} + f \kk\times \vv + \nabla \Phi_s 
+\nabla \int_p^{p_1} \frac{R}{p'}\left(\frac{p'}{p_0}\right)^{R/c_p}\theta \dd p'  + \A_{\vv} \vv = \boldsymbol S_{\vv},\\
\label{eq2.1.8b}&\frac{\de \theta}{\de t}+\vv \cdot \nabla \theta +\omega \frac{\de \theta}{\de p}+\A_\theta \theta
 \in \frac{L}{c_p } \left(\frac{p_0}{p}\right)^{R/c_p}D(\omega, \theta,q)+S_\theta,\\
\label{eq2.1.9b}&\frac{\de q}{\de t} + \vv\cdot\nabla q + \omega \frac{\de q}{\de p} + \A_q q \in 
-D(\omega, \theta,q)+S_q,
\end{align}
with the boundary conditions 
\begin{alignat}{3}
\label{eq2.1.10b}&\left(\nu_{\vv}\frac{\de \vv}{\de p} + \alpha_{\vv}\vv\right)\big|_{\Gamma_i} = 0,\quad & & \frac{\de \vv}{\de p}\big|_{\Gamma_u} = 0,\quad & & \vv\cdot \boldsymbol n\big|_{\Gamma_\ell}=\frac{\de (\vv\cdot\boldsymbol\tau)}{\de \boldsymbol n}\big|_{\Gamma_\ell}=0,\\
\label{eq2.1.11b} &\left(\nu_T\frac{\de \theta}{\de p}+ \alpha_\theta \theta\right)\big|_{\Gamma_i} =0,\qquad & & \left(\frac{\de \theta}{\de p}+\frac{R}{c_pp_0}\theta\right)\big|_{\Gamma_u} =0,\qquad & & \frac{\de \theta}{\de \boldsymbol n}\big|_{\Gamma_\ell}=0,\\
\label{eq2.1.12b}&\left(\nu_q\frac{\de q}{\de p}+\alpha_q q\right)\big|_{\Gamma_i}=0,\quad & & \frac{\de q}{\de p}\big|_{\Gamma_u}=0,\quad  & & \frac{\de q}{\de \boldsymbol n}\big|_{\Gamma_\ell} =0,
\end{alignat}
and the initial conditions
\begin{align}
\label{eq2.1.13b}\vv(x,y,p,0) &= \vv_0(x,y,p),\\
\label{eq2.1.14b}\theta(x,y,p,0) &= \theta_0(x,y,p),\\
\label{eq2.1.15b}q(x,y,p,0)&=q_0(x,y,p).
\end{align}

\subsection{Function spaces}\label{subsec2.3fs}
Here and throughout this article, we will not 
distinguish the notations for vector and scalar function spaces 
whenever they are self-evident from the context.
Denote by $H^s=H^s(\M)$ the classical Sobolev 
spaces of order $s$ on $\M$, and by $L^p=L^p(\M)$  ($1\leq p\leq\infty$) 
the classical $L^p$-Lebesgue space with norm $\|\cdot\|_{L^p}$. The only
exceptions will be made for the space $L^2$, whose norm will be written
with the single bar $|\cdot|$ and the scalar product as $(\cdot,\cdot)$, and for the space $H^1$,
whose scalar product is defined as
$$
((\varphi,\tilde\varphi))=(\nabla \varphi,\nabla \tilde\varphi)+ \int_\M \left(\frac{gp}{R\overline T}\right)^2\de_p\varphi\,\de_p\tilde\varphi\,\dd\M 
+\int_{\Gamma_i}\left(\frac{gp_1}{R\overline T}\right)^2 \varphi\,\tilde\varphi\,\dd\Gamma_i.
$$
and its norm denoted by 
$$
\|\varphi\|=((\varphi,\varphi))^{1/2}.
$$
Using the generalized Poincar\'e inequality (see e.g.~\cite[pp.~49-50]{Tem97}), the norm $\|\cdot\|$
is equivalent to the $H^1$-norm.
Regarding the velocity field $\vv$, we introduce the space
$$
\mathcal{V} = 
\big\{ \vv\in C^\infty(\M;\R^2)\,:\, \nabla\cdot\int_{p_0}^{p_1} \vv(x,y,p^\prime)\dd p^\prime=0,\,\vv\text{ satisfies }\eqref{eq2.1.10b}\big\},
$$
along with the $L^2$ and $H^1$-like spaces
\begin{align*}
&\H=\text{ The closure of }\mathcal{V}\text{ with respect to the norm of } (L^2)^2,\\
&\V=\text{ The closure of }\mathcal{V}\text{ with respect to the norm of } (H^1)^2.
\end{align*}
Due to the boundary conditions we consider, the space $\V$ is endowed with the scalar product
$$
((\vv,\tilde\vv))=(\nabla \vv,\nabla \tilde\vv)+ \int_\M \left(\frac{gp}{R\overline T}\right)^2\de_p\vv\cdot\de_p\tilde\vv\,\dd\M 
+\int_{\Gamma_i}\left(\frac{gp_1}{R\overline T}\right)^2 \vv\cdot\tilde\vv\,\dd\Gamma_i.
$$
Again, we will not differentiate the notation of norms and scalar products 
between spaces of vector-valued functions and spaces of scalar functions. Therefore, $((\cdot,\cdot))$ and $\|\cdot\|$
will denote the scalar product and the norm in $\V$ as well. The spaces of interest for the triplet
$(\vv,T,q)$ are therefore
$$
H=\H\times L^2\times L^2 \qquad\text{and}\qquad V=\V\times H^1\times H^1.
$$
Also, we shall make use of the space 
$$
W=V\cap (H^2)^4,
$$
when dealing with regularity of solutions.  We then have the following Gelfand-Lions inclusions
$$
W\subset V\subset H\subset V'\subset W',
$$
with continuous injections and each space being dense in the next one. 

\subsection{Bilinear and trilinear forms}
Having in mind the boundary conditions \eqref{eq2.1.10b}--\eqref{eq2.1.12b}, we observe the following:
if $(\vv,\theta,q), (\tilde\vv,\tilde\theta,\tilde q)\in V$, then integration by parts yields 
$$
\la\A_{\vv}\vv,\tilde{\vv}\ra=\mu_{\vv} (\nabla \vv,\nabla \tilde\vv)+ 
\nu_{\vv}\int_\M \left(\frac{gp}{R\overline T}\right)^2\de_p\vv\cdot\de_p\tilde\vv\dd\M 
+\alpha_{\vv}\int_{\Gamma_i}\left(\frac{gp_1}{R\overline T}\right)^2 \vv\cdot\tilde\vv\dd\Gamma_i,
$$
and
$$
\int_\M\nabla\Phi_s(x,y)\cdot\vv\,\dd\M=0.
$$
Similarly,
\begin{align*}
\la \A_\theta \theta,\tilde\theta\ra&= \int_\M \mu_T \nabla\theta\cdot\nabla\tilde\theta+\nu_T \left(\frac{gp}{R\overline T}\right)^2\de_p \theta\de_p\tilde\theta\,\dd\M + \alpha_T\int_{\Gamma_i} \left(\frac{gp_1}{R\overline T}\right)^2\theta\tilde\theta\,\dd\Gamma_i \\
&\quad- \nu_T\int_\M \left(\frac{g}{c_p\overline T}\right)^2\theta\tilde\theta\,\dd\M
 +\frac{\nu_T}{c_p}\int_\M \frac{g^2p}{R\overline{T}^2}\Big(\theta\de_p\tilde\theta-\tilde\theta\de_p\theta\Big)\, \dd\M,
\end{align*}
 and
 $$
\la \A_q q,\tilde q\ra=\mu_q (\nabla q,\nabla \tilde q)
+\nu_q \int_\M \left(\frac{gp}{R\overline T}\right)^2 \de_p q\,\de_p\tilde q\dd\M+\alpha_q\int_{\Gamma_i}\left(\frac{gp_1}{R\overline T}\right)^2q\tilde q\dd\Gamma_i.
$$
For $(\vv,\theta,q),(\tilde\vv,\tilde\theta,\tilde q), (\vv^\#,\theta^\#,q^\#)\in V$, we define the bilinear and trilinear forms as follows:
\begin{align*}
&a_{\vv}(\vv,\tilde{\vv}) =\mu_{\vv} (\nabla \vv,\nabla \tilde\vv)+ 
\nu_{\vv}\int_\M \left(\frac{gp}{R\overline T}\right)^2\de_p\vv\cdot\de_p\tilde\vv\dd\M 
+\alpha_{\vv}\int_{\Gamma_i}\left(\frac{gp_1}{R\overline T}\right)^2 \vv\cdot\tilde\vv\dd\Gamma_i,\\
&a_T(\theta,\tilde\theta) =  \mu_T (\nabla\theta,\nabla\tilde\theta)+\nu_T \int_\M\left(\frac{gp}{R\overline T}\right)^2\de_p \theta\de_p\tilde\theta\,\dd\M + \alpha_T\int_{\Gamma_i} \left(\frac{gp_1}{R\overline T}\right)^2\theta\tilde\theta\,\dd\Gamma_i, \\
&a_q(q,\tilde q) =\mu_q (\nabla q,\nabla \tilde q)
+\nu_q \int_\M \left(\frac{gp}{R\overline T}\right)^2 \de_p q\,\de_p\tilde q\dd\M+\alpha_q\int_{\Gamma_i}\left(\frac{gp_1}{R\overline T}\right)^2q\tilde q\dd\Gamma_i,\\
&m_{\theta}(\theta,\tilde\theta) = - \nu_T\int_\M \left(\frac{g}{c_p\overline T}\right)^2\theta\tilde\theta\,\dd\M+ \frac{\nu_T}{c_p}\int_\M \frac{g^2p}{R\overline{T}^2}\big(\theta\de_p\tilde\theta-\tilde\theta\de_p\theta\big)\, \dd\M, \\
&b(\vv,\tilde{\vv}, \vv^\#)=\int_\M\bigg( (\vv\cdot \nabla)\tilde{\vv} + \omega \frac{\de \tilde{\vv}}{\de p} \bigg)\cdot\vv^\#\,\dd\M,\\
&e_p(\theta,\tilde\vv)=\int_\M \nabla \int_p^{p_1} \frac{R}{p'}\left(\frac{p'}{p_0}\right)^{R/c_p}\theta \dd p' \cdot\tilde\vv\,\dd\M,\\
&e_c(\vv,\tilde\vv) = \int_\M (f\kk \times \vv)\cdot \tilde{\vv}\, \dd\M.
\end{align*}
Again, in order to keep the notation as simple as possible, the trilinear form $b$ could also have
scalar functions in the last two arguments, meaning, for example, that
$$
b(\vv,\tilde{\theta}, \theta^\#)=\int_\M\bigg( \vv\cdot \nabla\tilde{\theta} + \omega \frac{\de \tilde{\theta}}{\de p} \bigg)\theta^\#\,\dd\M.
$$
It is clear that the $a_i$'s ($i=\vv,\theta,q$) are  bilinear continuous symmetric forms on $H^1$ that satisfy
\begin{equation}\label{eq2.3.3}
|a_{\vv}(\vv,\tilde\vv)\leq C \|\vv\|\|\tilde\vv\|,\quad a_{T}(\theta,\tilde\theta)\leq C \|\theta\|\|\tilde\theta\|, 
\quad a_{q}(q, \tilde q)\leq C\|q\|\|\tilde q\|,\quad(C>0),
\end{equation}
for every $(\vv,\theta,q), (\tilde\vv,\tilde\theta,\tilde q)\in V$.
Furthermore,  they are coercive, for every $(\vv,\theta,q)\in V$:
\begin{equation}\label{eq2.3.4}
a_{\vv}(\vv,\vv)\geq \kappa_{\vv} \|\vv\|^2,\quad a_{T}(\theta,\theta)\geq \kappa_{T} \|\theta\|^2, 
\quad a_{q}(q,q)\geq \kappa_{q} \|q\|^2,
\end{equation}
where $\kappa_{\star}$ for $\star=\vv, T, q$ is defined by
\[
\kappa_{\star}=\min(\mu_{\star}, \nu_{\star}, \alpha_{\star}) >0.
\]
We also have that $m_\theta$ is bilinear continuous on $H^1$, with 
\begin{equation}\label{eq2.3.5}
m_\theta (\theta,\tilde \theta) \leq C\|\theta\|\|\tilde \theta\|, \quad (C>0),
\end{equation}
and satisfies
\begin{equation}\label{eq2.3.5b}
m_\theta (\theta,\theta) \geq -C|\theta|^2,  \quad (C>0).
\end{equation}
Finally, $e_{p}$ is bilinear continuous on  either 
$H^1\times \H$ or $L^2\times \V$, and  $e_c$ is bilinear continuous on $L^2\times L^2$. Hence, we have,
for some $C>0$,
\begin{equation}\label{eq2.3.6}
|e_p(\theta,\vv)|\leq C\|\theta\||\vv|,\quad |e_p(\theta,\vv)|\leq C|\theta|\|\vv\|,\qquad |e_c(\vv,\tilde\vv)|\leq C|\vv||\tilde\vv|.
\end{equation}
Also, we have the following result concerning the trilinear form $b$ (see \cite[Lemmas~2.1 and~3.1]{PTZ08}).

\begin{lemma}\label{lem2.1}
The trilinear form $b$ is continuous on $\V\times\V\times (\V\cap \H^2)$, and
\begin{equation}\label{eq2.3.7}
|b(\vv,\tilde\vv,\vv^\#)|\leq C\|\vv\||\tilde\vv|^{1/2}\|\tilde\vv\|^{1/2}\|\vv^\#\|_{H^2},
\end{equation}
for some constant $C>0$. Furthermore, 
\begin{equation}\label{eq2.3.9}
|b(\vv,\tilde\vv,\vv^\#)|\leq C\|\vv\|^{1/2}\|\vv\|^{1/2}_{H^2}\|\tilde\vv\|^{1/2}\|\tilde\vv\|_{H^2}^{1/2}|\vv^\#|,
\end{equation}
and
\begin{equation}\label{eq2.3.9b}
|b(\vv,\tilde\vv,\vv^\#)|\leq C\|\vv\|^{1/2}\|\vv\|^{1/2}_{H^2}\|\tilde\vv\||\vv^\#|^{1/2}\|\vv^\#\|^{1/2}.
\end{equation}
Also, using the incompressibility condition \eqref{eq1.1.3} 
and the boundary conditions $\uu\cdot  \boldsymbol n=0$ on $\partial\M$, we have
\begin{equation}\label{eq2.3.8}
b(\vv,\tilde\vv,\tilde\vv)=0.
\end{equation}
\end{lemma}

Using Lemma~\ref{lem2.1}, we then define the linear and bilinear continuous operators as follows:
\begin{alignat*}{3}
&A_{\vv}\colon\V\to \V', \qquad & &\la A_{\vv} \vv,\tilde\vv \ra=a_{\vv}(\vv,\tilde\vv),\qquad & &\forall  \vv,\tilde \vv\in \V,\\
&A_T\colon H^1\to (H^1)', \quad & &\la A_T \theta,\tilde\theta \ra=a_T(\theta,\tilde\theta),\quad & &\forall\, \theta,\tilde \theta\in H^1,\\
&A_q\colon H^1\to (H^1)', \quad & &\la A_q q,\tilde q \ra=a_q(q,\tilde q),\quad & &\forall\, q,\tilde q\in H^1,\\
&M_\theta\colon H^1\to (H^1)', \quad & &\la M_\theta \theta,\tilde\theta \ra=m_\theta(\theta,\tilde\theta),\quad & & \forall\, \theta,\tilde \theta\in H^1,\\
&E_p\colon L^2\to \V',\quad & &\la E_p\theta, \tilde\vv\ra=e_p(\theta,\tilde\vv),\quad & &\forall\, \theta \in L^2,\tilde\vv\in \V,\\
&E_c\colon\H\to \H,\quad  & &\la E_c\vv, \tilde\vv\ra=e_c(\vv,\tilde\vv),\quad & &\forall\, \vv,\tilde\vv\in \H.
\end{alignat*}
Also, from the properties of the trilinear form $b$, we can define a bilinear operator
$$
B\colon\V\times \V\to (\V\cap (H^2)^2)'
$$
acting as
$$
\la B(\vv,\tilde\vv),\vv^\#\ra=b(\vv,\tilde\vv,\vv^\#),\qquad \forall\,\vv,\tilde\vv\in \V, \vv^\#\in \V\cap (H^2)^2.
$$

\subsection{Quasi-strong solutions}\label{def:quasi-strong}
Let  $\vv_0\in \V$ and $\theta_0,q_0\in L^2$. 
Assume $(\boldsymbol S_{\vv},S_\theta,S_q)\in L^2(0,\T;H)$, $q_s\in(0,1)$ 
and $\T >0$. A vector $(\vv,\theta, q)$ is a \emph{quasi-strong} solution to \eqref{eq2.1.7b}--\eqref{eq2.1.15b}  if 
\begin{alignat*}{2}
&\vv\in C([0,\T];\V)\cap L^2(0,\T;H^2), \qquad & &\theta,q\in C([0,\T];L^2)\cap L^2(0,\T;H^1),\\
&\de_t\vv\in  L^2(0,\T;\H), \qquad & &  \de_t\theta,\de_t q\in  L^{2}(0,\T;(H^1)'),
\end{alignat*}
and, for almost every $t\in [0,t_1]$ and every $(\tilde\vv, \tilde\theta,\tilde q)\in \H\times H^1\times H^1$,
\begin{align}
&\la\de_t\vv, \tilde\vv\ra +a_{\vv}(\vv,\tilde\vv)+ b(\vv,\vv,\tilde\vv)+e_p(\theta,\tilde\vv)+e_c(\vv,\tilde\vv)=(\boldsymbol S_{\vv},\tilde\vv),\\
&\la\de_t\theta, \tilde\theta\ra +a_{\theta}(\theta,\tilde\theta)+ b(\vv,\theta,\tilde\theta)+m_\theta(\theta,\tilde\theta)
=\left( \frac{L}{c_p} \left(\frac{p_0}{p}\right)^{\frac{R}{c_p}} D(\omega, \theta,h_q),\tilde\theta\right)+(S_{\theta},\tilde\theta),\\
&\la\de_t q, \tilde q\ra +a_{q}(q,\tilde q)+ b(\vv,q,\tilde q)
=-(D(\omega, \theta,h_q),\tilde q)+(S_{q},\tilde q),
\end{align}
for some $h_q\in L^\infty(\M\times (0,\T))$ which satisfies the variational inequality
\begin{equation}\label{eq2.3.12}
([\tilde q-q_s]^+,1)- ([q-q_s]^+,1)\geq  ( h_q, \tilde q-q), \qquad \text{a.e. } t\in [0,\T], \quad \forall\, \tilde q\in H^1.
\end{equation}
Some remarks are in order.

\begin{remark}
With a little abuse of notation, we wrote
$$
D(\omega, \theta,h_q)=\frac{1}{p}\omega^{-}h_q\tF (p,\theta).
$$
Compared to \eqref{eq:DDDDD}, we now have that $D(\omega, \theta,h_q)$ is a single-valued map, denoting 
by $h_q$ an (arbitrary) element of the set $H(q-q_s)$.
\end{remark}

\begin{remark}
The variational inequality \eqref{eq2.3.12} expresses the fact that $h_q$ is an element
of the sub-differential of the positive part function $q\mapsto ([q-q_s]^+,1)$. Since
$$
\de ( [q-q_s]^+,1 )=H(q-q_s),
$$
it is easy to see that if $h_q\in H(q-q_s)$, then
\begin{equation}\label{amorino1}
h_q(x,y,p)=\begin{cases}
1, \qquad \text{if } q>q_s,\\
0, \qquad \text{if } q<q_s,
\end{cases}
\end{equation}
while if $q(x,y,p)=q_s$, we have
\begin{equation}\label{amorino2}
h_q(x,y,p)\in [0,1].
\end{equation}
\end{remark}

\section{An approximated problem}\label{sec:approx111}
\noindent In this section, we construct a family of regularized problems which approximate Problem \eqref{eq2.1.7b}--\eqref{eq2.1.15b}
in a suitable sense. In this way, the limit of such approximated solutions
will be shown to be a solution to our problem, in the sense made precise in Section \ref{def:quasi-strong}.
The proofs are based on \emph{a priori} estimates and compactness arguments, and the
variational inequality \eqref{eq2.3.12} plays an essential role.

\subsection{Problem \eqref{eq3.1.3}}
In order to introduce the approximated problems, we first define the real functions $H_\eps$ and $K_\eps$ approximating $H$ and $r^{+}$ (the positive part of $r$). Namely, for $\eps\in(0,1]$, let
\begin{equation}\label{eq:HepsK}
H_\eps(r)=
\begin{cases}
0, \quad &r\leq 0,\\
r/\eps, \quad &r\in (0,\eps],\\
1,\quad &r>\eps, 
\end{cases}
\qquad K_\eps(r)=
\begin{cases}
0, \quad &r\leq 0,\\
r^2/2\eps, \quad &r\in (0,\eps],\\
r-\eps/2,\quad &r>\eps.
\end{cases}
\end{equation}
It is straightforward to check that $K_\eps'=H_\eps$, 
\begin{equation}\label{eq:Hlip}
|H_\eps(r_1)|\leq 1,\quad
|H_\eps(r_1)-H_\eps(r_2)|\leq \frac{1}{\eps}|r_1-r_2|, \quad\quad \forall\, r_1,r_2\in \R,
\end{equation}
and
\begin{equation}\label{eq:Klip}
|K_\eps(r_1)-K_\eps(r_2)|\leq |r_1-r_2|, \quad\quad  \forall\, r_1,r_2\in \R.
\end{equation}
Moreover, 
\begin{equation}\label{eq:Kclos}
|K_\eps(r)-r|\leq \frac{\eps}{2}, \quad\quad  \forall\, r\geq 0.
\end{equation}
We then consider the following family of problems, depending on the parameter $\eps$, and for which we seek local quasi-strong solutions for every fixed $\eps>0$:
\begin{equation}\label{eq3.1.3}\tag{\textbf{P}$_\eps$}
\begin{aligned}
&\la\de_t\vv, \tilde\vv\ra +a_{\vv}(\vv,\tilde\vv)+ b(\vv,\vv,\tilde\vv)+e_p(\theta,\tilde\vv)+e_c(\vv,\tilde\vv)=(\boldsymbol S_{\vv},\tilde\vv),\\
&\la\de_t\theta, \tilde\theta\ra +a_{T}(\theta,\tilde\theta)+ b(\vv,\theta,\tilde\theta)+m_\theta(\theta,\tilde\theta)
=\left( \frac{L}{c_p} \left(\frac{p_0}{p}\right)^{\frac{R}{c_p}} D_\eps(\omega, \theta,q),\tilde\theta\right)+(S_{\theta},\tilde\theta),\\
&\la\de_tq, \tilde q\ra +a_{q}(q,\tilde q)+ b(\vv,q,\tilde q)
=-(D_\eps(\omega, \theta,q),\tilde q)+(S_{q},\tilde q).
\end{aligned}
\end{equation}
Here,
$$
D_\eps(\omega, \theta,q)=\frac{1}{p}\omega^{-}H_\eps(q-q_s)\tF (p,\theta),
$$
is now a well-defined map.

Now, we state the main result of this section, 
with the proof presented in the subsequent paragraphs.

\begin{theorem}\label{thm3.1}
Assume that $\vv_0\in \V$ and $\theta_0, q_0\in L^2$ are given and the forcing term 
$(\boldsymbol S_{\vv},S_\theta,S_q)\in L^2(0,\T;H)$. 
Let $q_s\in(0,1)$, $\T>0$ and $\eps>0$ be fixed. 
Then there exists $t_*>0$ ($t_*\leq\T$, \emph{independent of $\eps$}) and  
a  quasi-strong solution $(\vv^\eps,\theta^\eps, q^\eps)$ to \eqref{eq3.1.3} such that
$$
\vv^\eps\in C([0,t_*);\V)\cap L^2(0,t_*;H^2), \qquad \de_t \vv^\eps\in L^2(0,t_*;\H),
$$
and
$$ 
\theta^\eps,q^\eps\in C([0,t_*);L^2)\cap L^2(0,t_*;H^1), \qquad
\de_t\theta^\eps,\de_t q^\eps\in L^{2}(0,t_*;(H^1)').
$$
Moreover, the following estimates hold true:
\begin{equation}\label{eq:unifesteps}
	\begin{cases}
	\|\vv^\eps\|_{L^\infty(0,t_*; \V)}+\|\vv^\eps\|_{L^2(0,t_*; H^2)}+\|(\theta^\eps,q^\eps)\|_{L^\infty(0,t_*; L^2\times L^2)}\\
	\hspace{170pt}+\|(\theta^\eps,q^\eps)\|_{L^2(0,t_*;H^1\times H^1)} \leq \Q, \\
	\|\de_t \vv^\eps\|_{L^2(0,t_*;\H)}+\|(\de_t\theta^\eps,\de_t q^\eps)\|_{L^{2}(0,t_*;(H^1)'\times (H^1)')} \leq \Q,		
	\end{cases}
\end{equation}
where $\Q$ is a positive function independent of $\eps$ defined by
$$
\Q:=\Q (t_*,\|\vv_0\|,|\theta_0|,|q_0|,\|(\boldsymbol S_{\vv},S_\theta,S_q)\|_{L^2(0,\T;H)}),
$$
which is increasing in all its arguments.
\end{theorem}

\subsection{The Galerkin approximation}
In order to establish an existence result for this problem, we implement 
the Galerkin method using the eigenvectors $e_j$ of $A=A_{\vv}\oplus A_T \oplus A_q$:
\begin{equation}\label{eq:eigenvalues}
Ae_j=\lambda_je_j,\qquad j\geq 1,\quad 0<\lambda_1\leq\lambda_2\leq\cdots.
\end{equation}
The results of \cite[Section 4.1]{PTZ08} guarantee 
the following result.

\begin{lemma}\label{lem3.2}
The eigenvectors $e_j$ of $A$ belong to $(H^2)^4$. 
\end{lemma}

We denote by $A_{\vv}^{1/2}$ the square root of $A_{\vv}$; in particular,
\[
(A_{\vv}^{1/2} \vv, A_{\vv}^{1/2} \tilde\vv) = a_{\vv}(\vv, \tilde\vv),\qquad\forall\,\vv,\tilde\vv\in\V.
\]
Also denote by $A_{T}^{1/2}$ and $A_{q}^{1/2}$ the square roots of $A_{T}$ and $A_{q}$.

Note that the norm $|A_{\vv}^{1/2}\vv|$ is equivalent to the norm $\|\vv\|$ for $\vv\in \V$; 
also the norm $|A_{\vv}\vv|$ 
is equivalent to the norm $\|\vv\|_{H^2}$ for $\vv\in \V\cap (H^2)^2$. 
Similar results also hold for the operators $A_{\star}^{1/2}$ and $A_{\star}$ for $\star=T,q$ (cf.~\cite[Section 4]{PTZ08}).

\subsection{A priori $L^2$-estimates (I)}
Here we prove a
basic estimate on a Galerkin solution to \eqref{eq3.1.3}, which is contained in the following lemma.

\begin{lemma}\label{lem3.3.1}
Fix $\eps>0$ and let $(\vv^\eps,\theta^\eps,q^\eps)$ be a 
solution to \eqref{eq3.1.3} with initial datum  $(\vv_0,\theta_0,q_0)$. Then
\begin{equation}\label{eq3.3.10}
\begin{aligned}
\sup_{t\in [0,\T]}|(\vv^\eps,\theta^\eps,q^\eps)(t)|^2+ \int_0^{\T}\|(\vv^\eps,\theta^\eps,q^\eps)(t)\|^2\dd t 
&\leq C\e^{C\T}\big(|(\vv_0,\theta_0,q_0)|^2\\
&\quad+\norm{(\bs S_{\vv}, S_{T}, S_{q})}_{L^2(0,\T; H)}^2\big).
\end{aligned}
\end{equation}
\end{lemma}

\begin{proof}
In what follows, $C$ is an absolute constant \emph{independent} of $\eps$ and also \emph{independent} of the initial data $(\vv_0,\theta_0,q_0)$ and the forcing terms $(\boldsymbol S_{\vv},S_\theta,S_q)$. 
We start from the estimate on the velocity. Taking the $L^2$-scalar of the $\vv$-equation in 
\eqref{eq3.1.3} with $\vv^\eps$, we obtain the energy equation
$$
\frac12\frac{\dd}{\dd t} |\vv^\eps|^2 +a_{\vv}(\vv^\eps,\vv^\eps)=-e_p(\theta^\eps,\vv^\eps)-e_c(\vv^\eps,\vv^\eps)+(\boldsymbol S_{\vv},\vv^\eps).
$$
Each term of the right hand side above can be respectively estimated as
\begin{equation}\begin{split}\label{eq3.3.11}
&|e_p(\theta^\eps,\vv^\eps)|\leq C\|\theta^\eps\||\vv^\eps|\leq \frac{\kappa_{T}}{2}\|\theta^\eps\|^2+C|\vv^\eps|^2,\\
&e_c(\vv^\eps,\vv^\eps)=0,\\
&|(\boldsymbol S_{\vv},\vv^\eps)|\leq |\boldsymbol S_{\vv}||\vv^\eps|\leq C|\boldsymbol S_{\vv}|^2+\frac{\kappa_{\vv}}{2}|\vv^\eps|^2.
\end{split}\end{equation}
Therefore, we find
\begin{equation}\label{eq:firstest1}
\frac{\dd}{\dd t} |\vv^\eps|^2 +\kappa_{\vv}\|\vv^\eps\|^2\leq  C|\vv^\eps|^2+\kappa_{T}\|\theta^\eps\|^2+C|\boldsymbol S_{\vv}|^2 .
\end{equation}
Turning to the $\theta$-equation and applying the same reasoning, we find that
$$
\frac12 \frac{\dd}{\dd t}|\theta^\eps|^2 + a_T(\theta^\eps,\theta^\eps)+m_\theta(\theta^\eps,\theta^\eps)
=\left( \frac{L}{c_p} \left(\frac{p_0}{p}\right)^{\frac{R}{c_p}} D_\eps(\omega^\eps, \theta^\eps,q^\eps),\theta^\eps\right)+(S_{\theta},\theta^\eps),
$$
As noted in \eqref{eq2.3.5b}, 
we have $m_\theta(\theta^\eps,\theta^\eps)\geq -C|\theta^\eps|^2$. Moreover, 
\begin{align*}
\left( \frac{L}{c_p} \left(\frac{p_0}{p}\right)^{\frac{R}{c_p}} D_\eps(\omega^\eps, \theta^\eps,q^\eps),\theta^\eps\right)
&\leq |D_\eps(\omega^\eps,\theta^\eps,q^\eps)||\theta^\eps|\leq C|\omega^\eps||\theta^\eps|\\
&\leq C\|\vv^\eps\||\theta^\eps|\leq \frac{\kappa_{\vv}}{8}\|\vv^\eps\|^2+C|\theta^\eps|^2.
\end{align*}
From the trivial estimate
$$
|(S_{\theta},\theta^\eps)|\leq C|S_\theta|^2+\frac{\kappa_T}{2}|\theta^\eps|^2,
$$
we learn that
\begin{equation}\label{eq:firstest2}
\frac{\dd}{\dd t}|\theta^\eps|^2 + \kappa_{T}\|\theta^\eps\|^2\leq \frac{\kappa_{\vv}}{4}\|\vv^\eps\|^2+C|\theta^\eps|^2+C|S_\theta|^2.
\end{equation}
Finally, a similar estimate can be deduced for $q^\eps$. Indeed, from  \eqref{eq3.1.3} we find
$$
\frac12\frac{\dd}{\dd t}|q^\eps|^2+a_q(q^\eps,q^\eps)=-(D_\eps(\omega^\eps, \theta^\eps,q^\eps),q^\eps)+(S_{q},q^\eps).
$$
As before,
$$
|(D_\eps(\omega^\eps, \theta^\eps,q^\eps),q^\eps)|\leq  \frac{\kappa_{\vv}}{8}\|\vv^\eps\|^2+C|q^\eps|^2
$$
and
$$
|(S_{q},q^\eps)|\leq C|S_q|^2+\frac{\kappa_{q}}{2}|q^\eps|^2,
$$
so that
\begin{equation}\label{eq:firstest3}
\frac{\dd}{\dd t}|q^\eps|^2 + \kappa_{q}\|q^\eps\|^2\leq \frac{\kappa_{\vv}}{4}\|\vv^\eps\|^2+C|q^\eps|^2+C|S_q|^2.
\end{equation}
Adding together \eqref{eq:firstest1}, \eqref{eq:firstest2}, and \eqref{eq:firstest3} we obtain
\begin{equation}\label{eq:firstest4}
\begin{aligned}
\ddt\big[|\vv^\eps|^2+|\theta^\eps|^2+|q^\eps|^2\big] 
&+\kappa\bigl(\|\vv^\eps\|^2+\|\theta^\eps\|^2+\|q^\eps\|^2\bigr) \\
&\leq C\bigl(|\vv^\eps|^2+|\theta^\eps|^2+|q^\eps|^2\bigr)+C\bigl(|\boldsymbol S_{\vv}|^2+|S_\theta|^2+|S_q|^2\bigr),
\end{aligned}
\end{equation}
with $\kappa=\min\{\kappa_{\vv}/2,\kappa_{T},\kappa_{q}\}>0$.
The conclusion \eqref{eq3.3.10} follows from a standard application of the Gronwall lemma.
\end{proof}

\subsection{Change of equations for $\vv$}
In order to prove the $H^1$-regularity on the velocity, we first study the linear problem for the velocity in \eqref{eq3.1.3}$_1$ and show that the solution of the linear problem enjoys the $H^1$-regularity; then we prove the $H^1$ regularity of the solution for the nonlinear problem of \eqref{eq3.1.3}$_1$ in a short time. We state the problems in this subsection and establish the desired a priori estimates in the next subsection.

We write the equation \eqref{eq3.1.3}$_1$ in the functional form:
\begin{equation}\begin{cases}\label{eq3.4.0}
\displaystyle \frac{\dd \vv}{\dd t} + A_{\vv}\vv + B_{\vv}(\vv,\vv) + E_c\vv = \boldsymbol S_{\vv}-E_p\theta ,\\
\vv(0)= \vv_0,
\end{cases}\end{equation}
where the first equation is understood in $\big(\V\cap (H^2)^2\big)'$, the dual space of the $H^2$-like space for the velocity field, 
and $\theta$ is given in the space $L^2(0,\T;H^1)$ and hence $E_p\theta$ belongs to $L^2(0,\T;L^2)$.
The linear equation that we consider reads
\begin{equation}\begin{cases}\label{eq3.4.1}
\displaystyle \frac{\dd \vv^*}{\dd t} + A_{\vv}\vv^*  + E_c\vv^* = \boldsymbol S_{\vv} - E_p\theta,\\
\vv^*(0)= \vv_0.
\end{cases}\end{equation}
We then set $\vv'=\vv - \vv^*$, and by subtracting \eqref{eq3.4.1} from \eqref{eq3.4.0}, we see that $\vv'$ satisfies
\begin{equation}\begin{cases}\label{eq3.4.2}
\displaystyle\frac{\dd \vv'}{\dd t} + A_{\vv}\vv'  +B_{\vv}(\vv',\vv') + B_{\vv}(\vv',\vv^*)+B_{\vv}(\vv^*,\vv') + E_c\vv' =-B_{\vv}(\vv^*,\vv^*),\\
\vv'(0)=0.
\end{cases}\end{equation}
Our goal in the following subsection is to prove the a priori estimates for $\vv^*$ in \eqref{eq3.4.1} and $\vv'$ in \eqref{eq3.4.2}, and the existence and uniqueness of solutions $\vv$ for \eqref{eq3.4.0}.
\subsection{A priori $H^1$-estimates for the velocity (II)}
In this subsection, we are aiming to derive the $L^\infty(H^1)$ a priori estimates for the equations \eqref{eq3.4.0}--\eqref{eq3.4.2}. We start with the a priori estimate for $\vv^*$ in \eqref{eq3.4.1}.

\begin{lemma}\label{lem3.5.1}
Assume that $\vv_0\in \V,\theta\in L^2(0,\T;H^1)$, and $\boldsymbol S_{\vv}\in L^2(0,\T;L^2)$, and let $\vv^*$ be the solution to \eqref{eq3.4.1}. Then $\vv^*$ belongs to $L^\infty(0,\T;\V)\cap L^2(0,\T;H^2)$ and satisfies
\begin{equation}\label{eq3.5.2}
  \sup_{t\in [0,\T]}|A_{\vv}^{1/2}\vv^*(t)|^2+ \int_0^{\T}|A_{\vv}\vv^*(t)|^2\dd t 
  \leq
  C\e^{Ct_1}\left( \|\vv_0\|^2+\int_0^{\T}(|\boldsymbol S_{\vv}(t)|^2 +\|\theta(t)\|^2 )\dd t\right).
\end{equation}
\end{lemma}

\begin{proof}
Taking the $L^2$-scalar of \eqref{eq3.4.1}$_1$ with $\vv^*$ and using the estimates in \eqref{eq3.3.11}, we obtain
\begin{equation*}
  \frac{\dd}{\dd t} |\vv^*|^2 +\kappa_{\vv}\|\vv^*\|^2\leq  C\bigl(|\vv^*|^2+\|\theta\|^2\bigr)+|\boldsymbol S_{\vv}|^2,
\end{equation*}
which, by the Gronwall lemma, implies that
\begin{equation}\label{eq3.5.4}
  \sup_{t\in [0,\T]}|\vv^*(t)|^2+ \kappa_{\vv}\int_0^{\T}\|\vv^*(t)\|^2\dd t \leq C\e^{Ct_1}\left( |\vv_0|^2+\int_0^{\T}(|\boldsymbol S_{\vv}(t)|^2 +\|\theta(t)\|^2 )\dd t\right).
\end{equation}
We now multiply \eqref{eq3.4.1}$_1$ by $A_{\vv}\vv^*$ in $L^2$, and we find
\begin{equation}\begin{split}\label{eq3.5.5}
\frac12 \frac{\dd}{\dd t}|A_{\vv}^{1/2}\vv^*|^2 + |A_{\vv}\vv^*|^2 &= -e_p(\theta,{A_{\vv}\vv^*})-e_c(\vv^*,A_{\vv}\vv^*) - (\boldsymbol S_{\vv},A_{\vv}\vv^*)\\
&\leq C\norm{\theta}|A_{\vv}\vv^*| + C|\vv^*||A_{\vv}\vv^*| + |\boldsymbol S_{\vv}||A_{\vv}\vv^*|.
\end{split}\end{equation}
Using Young's inequality for the right-hand side of \eqref{eq3.5.5}, we arrive at
\begin{equation}
  \frac{\dd}{\dd t}|A_{\vv}^{1/2}\vv^*|^2 + |A_{\vv}\vv^*|^2 \leq C\big(\norm{\theta}^2 + |\vv^*|^2 +  |\boldsymbol S_{\vv}|^2\big),
\end{equation}
which, by the Gronwall lemma again, shows that
\begin{align*}
  \sup_{t\in [0,\T]}|A_{\vv}^{1/2}\vv^*(t)|^2+ \int_0^{\T}|A_{\vv}\vv^*(t)|^2\dd t &\leq C\e^{Ct_1}\|\vv_0\|^2\\
  &\quad+C\e^{Ct_1}\left( \int_0^{\T}(|\boldsymbol S_{\vv}(t)|^2 +|\vv^*(t)|^2+\|\theta(t)\|^2 )\dd t\right).
\end{align*}
The conclusion \eqref{eq3.5.2} then follows from \eqref{eq3.5.4}. 
\end{proof}

We now turn to the a priori estimate for $\vv'$ in \eqref{eq3.4.2}.

\begin{lemma}\label{lem3.5.2}
Assume that $\vv^*$ belongs to $L^\infty(0, \T; \V)\cap L^2(0, \T; H^2)$, 
and $\vv'$ is a solution to \eqref{eq3.4.2}. Then there exists $t_*>0$ ($t_*\leq \T$) such 
that $\vv'$ belongs to  $L^\infty(0,t_*;\V)\cap L^2(0,t_*;H^2)$ and satisfies
\begin{equation}\label{eq3.5.8}
  \sup_{t\in [0,t_*)}|A_{\vv}^{1/2}\vv'(t)|^2 + \int_0^{t_*}|A_{\vv}\vv'(t)|^2\dd t 
  \leq \mathcal{Q}_1\big(\norm{\vv^*}_{L^\infty(0,\T;\V) }, \norm{\vv^*}_{L^2(0,\T;H^2)}\big),
\end{equation}
where $\mathcal{Q}_1$ is a positive function.
\end{lemma}

\begin{proof}
Multiplying \eqref{eq3.4.2}$_1$ by $A_{\vv}\vv'$ in $L^2$, we obtain that
\begin{equation}\label{eq3.5.9}
\begin{split}
\frac{1}{2}\frac{\dd}{\dd t}|A_{\vv}^{1/2}\vv'|^2 + \abs{A_{\vv}\vv'}^2=&-e_c(\vv',A_{\vv}\vv') - b(\vv',\vv^*,A_{\vv}\vv') - b(\vv^*,\vv',A_{\vv}\vv')\\
&-b(\vv^*,\vv^*,A_{\vv}\vv')-b(\vv',\vv',A_{\vv}\vv').
\end{split}
\end{equation}
Using Young's inequality and \eqref{eq2.3.9} in Lemma~\ref{lem2.1} for the $b$-terms, 
we bound each term in the right-hand side of \eqref{eq3.5.9} as follows: using 
the fact that the norm $\abs{A_{\vv}\vv'}$ is equivalent to the norm $\norm{\vv'}_{H^2}$, and that 
the norm $|A_{\vv}^{1/2}\vv'|$ is equivalent to the norm $\norm{\vv'}$,
we obtain
\begin{equation}\label{eq3.5.9a}
\begin{split}
&|e_c(\vv',A_{\vv}\vv')| \leq \frac{1}{12}\abs{A_{\vv}\vv'}^2 + C\abs{\vv'}^2,\\
&|b(\vv',\vv^*,A_{\vv}\vv')| \leq C\norm{\vv'}^{1/2}\norm{\vv^*}^{1/2}\norm{\vv^*}_{H^2}^{1/2}\abs{A_{\vv}\vv'}^{3/2}\leq \frac{1}{12}\abs{A_{\vv}\vv'}^2 + C\norm{\vv'}^{2}\norm{\vv^*}^{2}\norm{\vv^*}_{H^2}^2,\\
&|b(\vv^*,\vv',A_{\vv}\vv')|\leq  \frac{1}{12}\abs{A_{\vv}\vv'}^2 + C\norm{\vv'}^{2}\norm{\vv^*}^{2}\norm{\vv^*}_{H^2}^2,\\
&|b(\vv^*,\vv^*,A_{\vv}\vv')|\leq C\norm{\vv^*}\norm{\vv^*}_{H^2}\abs{A_{\vv}\vv'}\leq\frac{1}{12}\abs{A_{\vv}\vv'}^2 + C\norm{\vv^*}^2\norm{\vv^*}_{H^2}^2,\\
&|b(\vv',\vv',A_{\vv}\vv')|\leq c_\star\norm{\vv'} \abs{A_{\vv}\vv'}^2.
\end{split}
\end{equation}
Taking all these bounds into account, we infer from \eqref{eq3.5.9} that
\begin{equation}\label{eq3.5.10}
\frac{\dd}{\dd t}|A_{\vv}^{1/2}\vv'|^2 +(1-c_\star|A_{\vv}^{1/2}\vv'|) \abs{A_{\vv}\vv'}^2 \leq \eta_1(t)|A_{\vv}^{1/2}\vv'|^2 + \eta_2(t),
\end{equation}
with
$$
\eta_1(t)=c+\eta_2(t),\quad\quad \eta_2(t) = c\norm{\vv^*}^{2}\norm{\vv^*}_{H^2}^2.
$$
As long as $1-c_\star|A_{\vv}^{1/2}\vv'|\geq 1/2$, that is
$$ 
|A_{\vv}^{1/2}\vv'|\leq \frac{1}{2c_\star},
$$
we then have, by Gronwall lemma and since $\vv'(0)=0$, on some interval of time $(0,t_*)$:
$$ 
|A_{\vv}^{1/2}\vv'|^2 \leq \int_0^t\eta_2(t) \e^{\int_\tau^t \eta_1(s)\dd s}\dd\tau,\quad \forall\,t\in(0,t_*).
$$
In fact, $t_*>0$ can be chosen as the minimum between $\T$ and $t_2$, where $t_2$ is either $+\infty$ or the first time at which
$$
\int_0^{t_2} \eta_2(s)\dd s = \frac{1}{4c_\star^2}\e^{-\int_0^{\T} \eta_1(s)\dd s}.
$$
In this way, we will then  have
$$
|A_{\vv}^{1/2}\vv'(t)|^2\leq \frac{1}{4c_\star^2},\quad\quad \forall\,t\in(0,t_*),
$$
and returning to \eqref{eq3.5.10} we find the estimate \eqref{eq3.5.8}.
\end{proof}

We now study the nonlinear equations \eqref{eq3.4.0} for $\vv$ and we have the following.

\begin{lemma}\label{lem3.5.3}
Assume that $\vv_0\in \V,\theta\in L^2(0,\T;H^1)$ and $\boldsymbol S_{\vv}\in L^2(0,\T;\H)$.
Then there exist $t_*>0$ ($t_*\leq \T$) and a unique solution $\vv$ to \eqref{eq3.4.0} 
which belongs to $L^\infty(0,t_*;\V)\cap L^2(0,t_*;H^2)$ and satisfies
\begin{equation}\label{eq3.5.11}
  \sup_{t\in [0,t_*)}|A_{\vv}^{1/2}\vv(t)|^2 + \int_0^{t_*}|A_{\vv}\vv(t)|^2\dd t 
  \leq \mathcal{Q}_2\big(t_*,\norm{\vv_0},\norm{\theta}_{L^2(0,\T;H^1)},\norm{\boldsymbol S_{\vv}}_{L^2(0,\T;L^2)}\big),
\end{equation}
where $\mathcal{Q}_2$ is a positive function, increasing in all arguments.
\end{lemma}

\begin{proof}
The existence of $\vv$ of \eqref{eq3.4.0} follows from the existence of $\vv^*$, solution 
to \eqref{eq3.4.1}, and $\vv'$ solution to \eqref{eq3.4.2}, based on the standard Galerkin approximation procedure.
Moreover estimate \eqref{eq3.5.11}  follows from the estimates \eqref{eq3.5.2} and \eqref{eq3.5.8}. 
We are left to prove the uniqueness.

Consider two solutions $\vv_1,\vv_2$ of \eqref{eq3.4.0} belonging to $L^\infty(0,t_*;\V)\cap L^2(0,t_*;H^2)$, and let $\tilde\vv=\vv_1-\vv_2$. Then $\tilde\vv$ satisfies
\begin{equation}\begin{cases}\label{eq3.5.12}
\displaystyle \frac{\dd \tilde\vv}{\dd t} + A_{\vv}\tilde\vv + B_{\vv}(\tilde\vv,\vv_1)+B(\vv_2,\tilde\vv) + E_c(\tilde\vv) =  0,\\
\tilde\vv(0)=0.
\end{cases}\end{equation}
Taking the $L^2$-scalar product between \eqref{eq3.5.12}$_1$ and $A_{\vv}\tilde\vv$ and using similar estimates 
in \eqref{eq3.5.9a} for the $b,e_c$ terms, we obtain
\begin{equation}
  \frac{\dd }{\dd t}|A_{\vv}^{1/2}\tilde\vv|^2 + |A_{\vv}\tilde\vv|^2 \leq \eta(t)|A_{\vv}^{1/2}\tilde\vv|^2,
\end{equation}
where 
$$
\eta(t)=c\big(1+\|\vv_1\|^2\|\vv_1\|_{H^2}^2 + \|\vv_2\|^2\|\vv_2\|_{H^2}^2\big)\in L^1(0,t_*),
$$
for some $c>0$. The uniqueness follows by the Gronwall lemma.
\end{proof}

\begin{remark}
We remark that the equations \eqref{eq3.4.0}--\eqref{eq3.4.2} are independent of $\eps$, and thus the estimates \eqref{eq3.5.2}, \eqref{eq3.5.8}, and \eqref{eq3.5.11} are independent of $\eps$ and so is the choice of $t_*$ in Lemmas \ref{lem3.5.2}--\ref{lem3.5.3}.
\end{remark}

\subsection{Local well-posedness of \eqref{eq3.1.3}}
With the a priori estimates in Lemmas~\ref{lem3.3.1}--\ref{lem3.5.3} at hand, we are now ready to prove Theorem~\ref{thm3.1}.
\begin{proof}[Proof of Theorem~\ref{thm3.1}]
Instead of solving \eqref{eq3.1.3}, we solve \eqref{eq3.1.3}$_{2,3}$ coupled (through $\theta$) with 
\eqref{eq3.4.1}--\eqref{eq3.4.2}. We apply the standard Galerkin approximation procedure for the 
unknowns $(\vv^*,\vv',\theta,q)$ using the eigenvectors of $A$ introduced in \eqref{eq:eigenvalues}. 
The a prori estimates \eqref{eq3.3.10}, \eqref{eq3.5.2}, and \eqref{eq3.5.8} show that we have uniform 
bounds independent of $\eps$. It is then straightforward to pass to the limit, and we obtain the existence 
of $(\vv^\eps{^*},\vv^\eps{'},\theta^\eps,q^\eps)$ for \eqref{eq3.4.1}--\eqref{eq3.4.2}, \eqref{eq3.1.3}$_{2,3}$ 
and hence the existence of $(\vv^\eps,\theta^\eps,q^\eps)$ for \eqref{eq3.1.3} by letting $\vv^\eps=\vv^\eps{^*}+\vv^\eps{'}$. 

The first estimate \eqref{eq:unifesteps}$_1$ then follows from the a priori estimates \eqref{eq3.3.10} and \eqref{eq3.5.11}.
We are left to prove the second estimate \eqref{eq:unifesteps}$_2$ on the time-derivates. Given $\tilde\vv\in L^2(0,t_*;\H)$ 
with $\|\tilde\vv\|_{L^2(0,t_*;\H)}\leq 1$, from
\eqref{eq3.1.3}$_1$ we infer that
$$
|(\de_t\vv^\eps,\tilde\vv)|\leq C\big[ \|\vv^\eps\|_{H^2}|\tilde\vv|+\|\vv^\eps\|\|\vv^\eps\|_{H^2}|\tilde\vv|+\|\theta^\eps\||\tilde\vv|
+|\vv^\eps||\tilde\vv|+|\boldsymbol S_{\vv}||\tilde\vv|\big].
$$
Integrating in time on $(0,t_*)$ and using Young's inequality then gives
$$
\de_t\vv^\eps\in L^2(0,t_*;\H),
$$
thanks to the estimate  \eqref{eq:unifesteps}$_1$. Regarding $\de_t\theta^\eps$, we take a test function 
$\tilde\theta\in L^2(0,t_*;H^1)$ with norm at most 1. Thanks to \eqref{eq2.3.9b} and \eqref{eq2.3.8},
we obtain from \eqref{eq3.1.3}$_2$ that
\begin{align*}
|\la\de_t\theta^\eps,\tilde\theta\ra|&\leq C\big[ \|\theta^\eps\|\|\tilde\theta\|+\|\vv^\eps\|^{1/2}\|\vv^\eps\|^{1/2}_{H^2}|\theta^\eps|^{1/2}\|\theta^\eps\|^{1/2}\|\tilde\theta\|
+\|\vv^\eps\||\tilde\theta|+|S_{\theta}||\tilde\theta|\big]\\
&\leq C\big[ \|\theta^\eps\|^2(1+|\theta^\eps|^2)+\|\vv^\eps\|^2(1+\|\vv^\eps\|^2_{H^2})+|S_{\theta}|^2+\|\tilde\theta\|^2\big].
\end{align*}
Similarly as for $\de_t\vv^\eps$, we integrate in time on $(0,t_*)$ and use Young's inequality; we arrive at
$$
\de_t\theta^\eps\in L^{2}(0,t_*;(H^1)').
$$
The argument for $\de_t q^\eps$ can be repeated word for word, allowing us to conclude the 
estimates \eqref{eq:unifesteps}$_2$. As we have already observed, all these estimates are actually implemented 
using  a Galerkin method based on the eigenvectors \eqref{eq:eigenvalues} of $A$, the proof of 
Theorem~\ref{thm3.1} is thus concluded.
\end{proof}

\section{Local existence of quasi-strong solutions}\label{sec:localquasi}

\noindent We provide here a proof of the local  existence of quasi-strong solutions of the $(\vv,\theta,q)$ system by passing
to the limit in the approximated problem \eqref{eq3.1.3} as $\eps\to 0$. The treatment of the potential temperature and the specific 
humidity equations resembles that of \cite{CT12}. As a straightforward consequence, we deduce a local existence 
result for the $(\vv,T,q)$ system as well.

\subsection{Passage to the limit as $\eps \to 0$}
Thanks to the fact that the estimates \eqref{eq:unifesteps} in Theorem \ref{thm3.1} are uniform in $\eps$ (the bounds are independent of $\eps$),
we infer the existence of a triplet $(\vv,\theta, q)$ such that 
$$
	\vv\in C([0,t_*);\V)\cap L^2(0,t_*;H^2), \qquad \de_t \vv\in L^2(0,t_*;\H),
$$	
and
$$ 
\theta,q\in C([0,t_*);L^2)\cap L^2(0,t_*;H^1), \qquad
\de_t\theta,\de_t q\in L^{2}(0,t_*;(H^1)'),
$$
for which the following convergences up to not relabeled subsequences  are true. As customary, 
$\to$, $\rightharpoonup$, and $\stackrel{*}{\rightharpoonup}$ indicate strong, weak, and weak-$*$ convergence as $\eps\to 0$, respectively:

\begin{itemize}
	\item $\vv^\eps\rightharpoonup \vv$  in $L^2(0,t_*;H^2)$ and $\de_t \vv^\eps \rightharpoonup \de_t \vv$ in $L^2(0,t_*;\H)$.
	As a consequence (see \cite{Lio69}), $\vv^\eps\to \vv$ in $L^2(0,t_*;H^{3/2})$.
	\item $\omega^\eps\to \omega$ in $L^2(0,t_*;H^{1/2})$, which follows from the expression \eqref{eq2.1.5} for $\omega$.
	\item $ (\theta^\eps,q^\eps)\rightharpoonup  (\theta,q)$ in $L^2(0,t_* ;H^1\times H^1)$ and
	$ \de_t(\theta^\eps,q^\eps)\rightharpoonup  \de_t (\theta,q)$ in $L^{2}(0,t_* ;(H^1)'\times (H^1)')$. Therefore,
	$(\theta^\eps,q^\eps)\to  (\theta,q)$ in $L^2(0,t_* ;L^2\times L^2)$.
	
	\item $H_\eps(q^\eps-q_s)\stackrel{*}{\rightharpoonup} h_q$ in $L^\infty(\M\times[0,t_*])$.
\end{itemize}
By interpolation, we also have  $\vv\in C([0,t_*);\V)$.
We now consider test functions of the form $(\tilde\vv,\tilde\theta,\tilde q)\varphi(t)$, where 
$(\tilde\vv,\tilde\theta,\tilde q)\in \H\times H^1\times H^1$
and $\varphi$ in $C^1([0,t_*])$ is a scalar function such that $\varphi(t_*)=0$. 
We take the $L^2$-scalar product for \eqref{eq3.1.3} 
with $(\tilde\vv,\tilde\theta,\tilde q)\varphi(t)$, integrate in time from $0$ to $t_*$ and integrate by parts for the first term, to arrive at
$$
\begin{aligned}
&\int_0^{t_*}-\la\vv^\eps, \tilde\vv\ra \de_t\varphi + \bigl(a_{\vv}(\vv^\eps,\tilde\vv) + b(\vv^\eps,\vv^\eps,\tilde\vv)+e_p(\theta^\eps,\tilde\vv)+e_c(\vv^\eps,\tilde\vv)\bigr)\varphi \,\dd t\\
&\qquad=\la \vv_0,\tilde\vv\ra\varphi(0)+\int_0^{t_*}(\boldsymbol S_{\vv},\tilde\vv)\varphi\,\dd t,\\
&\int_0^{t_*}-\la\theta^\eps, \tilde\theta\ra \de_t\varphi + \bigl(a_{\theta}(\theta^\eps,\tilde\theta)+ b(\vv^\eps,\theta^\eps,\tilde\theta)+m_\theta(\theta^\eps,\tilde\theta)\bigr)\varphi\, \dd t \\
&\qquad=\la \theta_0,\tilde\theta\ra\varphi(0)+
\int_0^{t_*}\left(\left( \frac{L}{c_p} \left(\frac{p_0}{p}\right)^{\frac{R}{c_p}} D_\eps(\omega^\eps, \theta^\eps,q^\eps),\tilde\theta\right)+(S_{\theta},\tilde\theta)\right)\varphi \,\dd t,\\
&\int_0^{t_*}-\la q^\eps, \tilde q\ra \de_t\varphi +\bigl(a_{q}(q^\eps,\tilde q)+ b(\vv^\eps,q^\eps,\tilde q)\bigr)\varphi\,\dd t\\
&\qquad =\la q_0,\tilde q\ra\varphi(0) +\int_0^{t_*}\bigl(-(D_\eps(\omega^\eps, \theta^\eps,q^\eps),\tilde q)+(S_{q},\tilde q)\bigr)\varphi \,\dd t,
\end{aligned}
$$
The only problematic terms are the nonlinear ones, as the linear terms converge to their corresponding limits in a straightforward manner due to the above convergences. We start with the term
$$
\int_0^{t_*}\left( \frac{L}{c_p} \left(\frac{p_0}{p}\right)^{\frac{R}{c_p}} D_\eps(\omega^\eps, \theta^\eps,q^\eps),\tilde\theta\right)\varphi \,\dd t.
$$
We have
\begin{align*}
&\int_0^{t_*}\left( \frac{L}{c_p} \left(\frac{p_0}{p}\right)^{\frac{R}{c_p}} D_\eps(\omega^\eps, \theta^\eps,q^\eps),\tilde\theta\right)\varphi \,\dd t
-\int_0^{t_*}\left( \frac{L}{c_p} \left(\frac{p_0}{p}\right)^{\frac{R}{c_p}} D(\omega, \theta,h_q),\tilde\theta\right)\varphi \,\dd t\\
&\quad=\int_0^{t_*}\left( \frac{L}{c_p} \left(\frac{p_0}{p}\right)^{\frac{R}{c_p}} \omega^\eps{^-} H_\eps(q^\eps-q_s) (\tF (p,\theta^\eps) -\tF (p,\theta) ),\tilde\theta\right)\varphi \,\dd t\\
&\qquad+\int_0^{t_*}\left( \frac{L}{c_p} \left(\frac{p_0}{p}\right)^{\frac{R}{c_p}} (\omega^\eps{^-}-\omega^-) H_\eps(q^\eps-q_s) \tF (p,\theta),\tilde\theta\right)\varphi \,\dd t\\
&\qquad+\int_0^{t_*}\left( \frac{L}{c_p} \left(\frac{p_0}{p}\right)^{\frac{R}{c_p}} \omega^- (H_\eps(q^\eps-q_s) - h_q)\tF (p,\theta),\tilde\theta\right)\varphi \,\dd t\\
&\qquad=J_1+J_2+J_3.
\end{align*}
Using the boundedness of $H_\eps$ (see \eqref{eq:Hlip}) and 
the Lipschitz condition on $\tF(p,\cdot)$ (see \eqref{eq:tF}), we bound the term $J_1$ as
\begin{equation}
\begin{split}
\abs{J_1} &\leq c\int_0^{t_*}\int_\M |\omega^\eps| |\theta^\eps-\theta| |\tilde\theta| \,\dd\M \dd t \\
&\leq c\int_0^{t_*}\|\omega^\eps(t)\|_{L^3} |\theta^\eps(t)-\theta(t)|\|\tilde\theta\|_{L^6}\, \dd t\\
&\leq c\|\tilde\theta\| \int_0^{t_*} \|\vv^\eps\|_{H^2} | \theta^\eps(t)-\theta(t)|  \,\dd t\\
&\leq c \norm{\vv^\eps}_{L^2(0,t_*;H^2)}\norm{\theta^\eps-\theta}_{L^2(0,t_*;L^2)} \|\tilde\theta\|,
\end{split}
\end{equation}
which converges to $0$ as $\eps\to 0$ thanks to the boundedness of $\vv^\eps$ and the strong convergence of $\theta^\eps$. 
As a preliminary, using the Sobolev embedding $H^{1/2}\subset L^3$, we obtain 
\begin{equation}\label{eq4.4.e1}
	\omega^\eps\rightarrow \omega,\text{ strongly in }L^2(0,t_*;L^3),
\end{equation}
and thus also in $L^2(0,t_*;L^2)$.
Now, for the second term $J_2$, using the boundedness of $H_\eps$ and $\tF$ (see \eqref{eq:Hlip} and \eqref{eq:tF}), we obtain
\begin{equation}
|J_2| \leq c\int_0^{t_*}\int_\M | \omega^\eps{^-}-\omega^- ||\tilde\theta|  \,\dd \M \dd t\leq c\norm{\omega^\eps{^-}-\omega^-  }_{L^2(0,t_*;H)}|\tilde\theta|,
\end{equation}
which converges to $0$ as $\eps\to 0$ by using \eqref{eq4.4.e1}. 
For the last term $J_3$, we observe that $\omega^-\tF(p,\theta)\tilde\theta\varphi$ belongs to $L^1(\M\times(0,t_*))$. Hence, the weak-$*$ convergence of $H_\eps(q_\eps-q_s)$ is enough to pass to the limit.  
Therefore, as $\eps\to 0$,
$$
\int_0^{t_*}\left( \frac{L}{c_p} \left(\frac{p_0}{p}\right)^{\frac{R}{c_p}} D_\eps(\omega^\eps, \theta^\eps,q^\eps),\tilde\theta\right)\varphi \,\dd t\to
\int_0^{t_*}\left( \frac{L}{c_p} \left(\frac{p_0}{p}\right)^{\frac{R}{c_p}} D(\omega, \theta,h_q),\tilde\theta\right)\varphi \,\dd t.
$$
The analogous term in the $q$-equation converges in the same exact way.

We now turn to the trilinear term $b$. Considering the typical most problematic term, we have that, as $\eps\to 0$,
$$
\int_0^{t_*}\int_\M \omega^\eps\de_p \vv^\eps\tilde\vv\varphi\,\dd\M \dd t\to 
\int_0^{t_*}\int_\M \omega\de_p \vv\tilde\vv\varphi\,\dd\M \dd t,
$$
where we used \eqref{eq4.4.e1} and that $\de_p\vv^\eps$ converges weakly to $\de_p\vv$ in $L^2(0,t_*; L^6)$ by the Sobolev embedding $H^1\subset L^6$.
The other terms in $b$ are similar or simpler. Therefore, we conclude that, as $\eps\to 0$,
$$
\int_0^{t_*}b(\vv^\eps,\vv^\eps,\tilde\vv)\varphi \,\dd t\to \int_0^{t_*}b(\vv,\vv,\tilde\vv)\varphi \,\dd t.
$$
We are left with the trilinear form $b$ involving $\theta$ (the one involving $q$ is exactly the same). 
Similarly, we consider the typical most problematic term, and hence, it is enough to show that
$$
\int_0^{t_*}\int_\M \omega^\eps\de_p \theta^\eps\tilde\theta\varphi\,\dd\M \dd t\to 
\int_0^{t_*}\int_\M \omega\de_p \theta\tilde\theta\varphi\,\dd\M \dd t.
$$
We have
\begin{align*}
&\int_0^{t_*}\int_\M \bigl(\omega^\eps\de_p \theta^\eps-\omega\de_p \theta\bigr)\tilde\theta\varphi\,\dd\M \dd t \\
&\quad=\int_0^{t_*}\int_\M \bigl(\omega^\eps-\omega\bigr)\de_p \theta^\eps \tilde\theta\varphi\,\dd\M \dd t+
\int_0^{t_*}\int_\M \omega\bigl(\de_p \theta^\eps-\de_p \theta\bigr)\tilde\theta\varphi\,\dd\M \dd t,
\end{align*}
with 
\begin{align*}
\left|\int_0^{t_*}\int_\M \bigl(\omega^\eps-\omega\bigr)\de_p \theta^\eps \tilde\theta\varphi\,\dd\M \dd t\right|&\leq
c\int_0^{t_*}\|\omega^\eps-\omega\|_{L^3}|\de_p \theta^\eps|\|\tilde\theta\|\dd t\\
&\leq c   \|\omega^\eps-\omega\|_{L^2(0,t_*;L^3)} \|\theta^\eps\|_{L^2(0,t_*;H^1)}\|\tilde\theta\|\to 0\qquad (\eps\to 0),
\end{align*}
and
\begin{align*}
\left|\int_0^{t_*}\int_\M \omega\bigl(\de_p \theta^\eps-\de_p \theta\bigr)\tilde\theta\varphi\,\dd\M \dd t\right|\to 0 \qquad (\eps\to 0),
\end{align*}
as $\de_p \theta^\eps\rightharpoonup \de_p \theta$ in $L^2(0,t_*;L^2)$ and $\omega \tilde\theta \varphi\in L^2(0,t_*;L^2)$.
Indeed, $\omega\in L^2(0,t_*;L^3)$ and $\tilde\theta \varphi\in L^\infty(0,t_*;L^6)$.

It remains to show that $h_q$ belongs to $H(q-q_s)$ in the weak sense specified by the variational inequality \eqref{eq2.3.12}.
This has been already proved in \cite{CT12}, but we sketch the argument here for 
the sake of completeness.  
For every $\eps>0$, we have an approximate variational inequality 
\begin{equation}\label{eq:approxineq}
\int_0^{t_*}( K_\eps(\tilde q-q_s),1)\dd t- \int_0^{t_*}( K_\eps(q_\eps-q_s),1)\dd t \geq \int_0^{t_*} \la H_\eps(q_\eps-q_s),\tilde q-q_\eps\ra\dd t,
\end{equation}
for each $\tilde q\in L^2(0,t_*; H^1)$, since $H_\eps(q_\eps-q_s)$ is the G\^ateaux derivative of the convex function 
$$
\int_0^{t_*}(K_\eps(\cdot),1)\dd t\colon L^2(0,t_*;V)\to \R 
$$
at the point $q_\eps-q_s$ (see \eqref{eq:HepsK}). From the weak-$*$ convergence $H_\eps(q_\eps-q_s) \stackrel{*}{\rightharpoonup} h_q$
in $L^\infty(\M\times [0,t_*])$ and the strong convergence $q_\eps\to q$ in $L^2(0,t_*; H)$ we find that 
$$
\int_0^{t_*}\la H_\eps(q_\eps-q_s),q_\eps-\tilde q\ra\dd t\to \int_0^{t_*}\la h_q,q-\tilde q\ra\dd t,\qquad \forall \tilde q\in L^2(0,t_*;V),
$$
as $\eps\to 0$.
Moreover, owing to \eqref{eq:Klip} and \eqref{eq:Kclos}, we observe that
\begin{align*}
 &\Big|\int_0^{t_*}( K_\eps(q_\eps-q_s),1)\dd t- \int_0^{t_*}( [q-q_s]^+,1)\dd t  \Big|   \\
 &\qquad \leq\int_0^{t_*}( |K_\eps(q_\eps-q_s)-K_\eps(q-q_s)|,1)\dd t+ \int_0^{t_*}( |K_\eps(q-q_s)-[q-q_s]^+|,1)\dd t  \\
 &\qquad \leq |\M|^{1/2}t_*^{1/2}\|q_\eps-q\|_{L^2(0,t_*;H)}+ \frac{\eps}{2}|\M|t_*.
\end{align*}
Therefore,
$$
\lim_{\eps\to 0} \int_0^{t_*}( K_\eps(q_\eps-q_s),1)\dd t= \int_0^{t_*}( [q-q_s]^+,1)\dd t  .
$$
From the calculation above, it is also clear that 
$$
\lim_{\eps\to 0} \int_0^{t_*}( K_\eps(\tilde q-q_s),1)\dd t= \int_0^{t_*}( [\tilde q-q_s]^+,1)\dd t , \qquad  \tilde q\in L^2(0,t_*;V).
$$
Consequently, we can pass to the limit as $\eps\to 0$ in \eqref{eq:approxineq}, concluding that
$$
\int_0^{t_*}\la [\tilde q-q_s)]^+,1\ra\dd t-\int_0^{t_*}\la [q-q_s]^+,1\ra\dd t \geq \int_0^{t_*} \la h_q,\tilde q-q\ra\dd t , \qquad \forall \tilde q\in L^2(0,t_*;V).
$$
Again, this implies in particular that
$$
([\tilde q-q_s]^+,1)- ([q-q_s]^+,1)\geq  \la h_q, \tilde q-q\ra,
$$
for every $\tilde q\in V$ and a.e. $t\in (0,t_*]$, as desired.
We have proved the following statement.

\begin{theorem}\label{thm3.2}
Let $\vv_0\in \V$ and $\theta_0,q_0\in L^2$.
Assume $(\boldsymbol S_{\vv},S_\theta,S_q)\in L^2(0,\T;H)$, $q_s\in(0,1)$ 
and $\T >0$.  There exists $t_*>0$ ($t_*\leq\T$) and  
a \emph{quasi-strong} solution to \eqref{eq2.1.7b}--\eqref{eq2.1.15b}  such that 
\begin{alignat*}{2}
&\vv\in C([0,t_*);\V)\cap L^2(0,t_*;H^2), \qquad & &\theta,q\in C([0,t_*);L^2)\cap L^2(0,t_*;H^1),\\
&\de_t\vv\in  L^2(0,t_*;\H), \qquad & &  \de_t\theta,\de_t q\in  L^{2}(0,t_*;(H^1)').
\end{alignat*}
\end{theorem}

\begin{remark}\label{rmk3.5.1}
By Lemma~\ref{lem3.3.1}, the quasi-strong solution $(\vv,\theta, q)$ to \eqref{eq2.1.7b}--\eqref{eq2.1.9b} satisfies the estimate
$$
\norm{ (\vv,\theta,q) }_{L^\infty(0,t_*; H)}^2 + \norm{(\vv,\theta,q)}_{L^2(0,t_*; V)}^2
\leq Ce^{C\T}\big(|(\vv_0,\theta_0,q_0)| +\norm{(\bs S_{\vv}, S_{\theta}, S_{q})}_{L^2(0,\T; H)}^2\big),
$$
which shows that we have a uniform bound for $\norm{(\vv,\theta,q)}_{H}$. This estimate will be very useful for obtaining the global strong solutions.
\end{remark}

\subsection{The $(\vv,T,q)$-system}
We now revert back to our original  $(\vv,T,q)$-system. The weak formulation 
reads 
\begin{align}
&\la\de_t\vv, \tilde\vv\ra +a_{\vv}(\vv,\tilde\vv)+ b(\vv,\vv,\tilde\vv)+e_p(T,\tilde\vv)+e_c(\vv,\tilde\vv)=(\boldsymbol S_{\vv},\tilde\vv),\label{eq:vTq1}\\
&\la\de_tT, \widetilde T\ra +a_{T}(T,\widetilde T)+ b(\vv,T,\widetilde T)
=m_T(\omega,T,\widetilde T)+\frac{L}{c_p}(  D(\omega, T,h_q),\widetilde T)+(S_{T},\widetilde T),\label{eq:vTq2}\\
&\la\de_t q, \tilde q\ra +a_{q}(q,\tilde q)+ b(\vv,q,\tilde q)
=-(D(\omega, \theta,h_q),\tilde q)+(S_{q},\tilde q)\label{eq:vTq3},
\end{align}
where
\begin{align*}
m_T(\omega,T,\widetilde T)=\int_\M \frac{Rc_p}{p}\omega T\widetilde T\dd\M,
\end{align*}
In order not to complicate the notation, we redefine
$$
e_p(T,\tilde\vv)=\int_\M \nabla \int_p^{p_1} \frac{R}{p'}T\, \dd p' \cdot\tilde\vv\,d\M
$$
and
$$
D(\omega, T,h_q)=\frac{1}{p}\omega^{-}h_qF (T).
$$
Using Theorem~\ref{thm3.2} and Remark \ref{rmk3.5.1} with the relation \eqref{eq2.1.1} between $T$ and $\theta$, we obtain the following.

\begin{corollary}\label{cor3.2}
Let there be given  $\vv_0\in \V$ and $T_0,q_0\in L^2$. 
Assume $(\boldsymbol S_{\vv},S_T,S_q)\in L^2(0,\T;H)$, $q_s\in(0,1)$ 
and $\T >0$.  There exists $t_*>0$ ($t_*\leq\T$) and  
a \emph{quasi-strong} solution to \eqref{eq:vTq1}--\eqref{eq:vTq3} such that 
\begin{alignat*}{2}
&\vv\in C([0,t_*);\V)\cap L^2(0,\T;H^2), \qquad & &T,q\in C([0,t_*);L^2)\cap L^2(0,t_*;H^1),\\
&\de_t\vv\in  L^2(0,t_*;\H), \qquad & &  \de_tT,\de_t q\in  L^{2}(0,t_*;(H^1)'),
\end{alignat*}
and the following $L^2$-estimate
\begin{equation}\label{eq4.4.66}
\norm{ (\vv,T,q) }_{L^\infty(0,t_*; H)}^2 + \norm{(\vv,T,q)}_{L^2(0,t_*; V)}^2
\leq Ce^{C\T}\big(|(\vv_0,T_0,q_0)| +\norm{(\bs S_{\vv}, S_{T}, S_{q})}_{L^2(0,\T; H)}^2\big),
\end{equation}
holds for some constant $C>0$ independent of initial data and the time $\T$ and $t_*$.
\end{corollary}

\section{Global strong solutions}\label{sec5}
\noindent The notion of quasi-strong solutions was introduced in the previous section in order
to deal with the lesser regularity of the vertical component $\omega$ of the velocity field
with respect to $\vv$. Moreover, the use of the potential temperature $\theta$ turned out to be convenient to obtain the basic $L^2$-$H^1$ estimates, circumventing the difficulty
 of dealing with the anti-dissipative term
 $$
 -\frac{R}{c_p p}\omega T,
 $$
 present in the equation for the temperature $T$. From here on, we only consider the $(\vv,T,q)$-system
 of equations, for which existence of quasi-strong solutions has been established in Corollary~\ref{cor3.2}.
 
The first goal of this section is to prove that, for more regular initial data, the local solutions derived in the previous section are in fact \emph{strong}. The second goal aims to show the existence of global \emph{strong} solutions by using Theorem~\ref{thm6.1.1} (see also \cite{KZ07}).
 
\subsection{Strong solutions}
We begin by defining the concept of strong solutions.

\begin{definition}\label{def:strong}
Let  $(\vv_0,T_0, q_0)$ be given in $V$. 
Assume $(\boldsymbol S_{\vv},S_T,S_q)$ are given in $L^2(0,\T;H)$, $q_s\in(0,1)$ 
and $\T >0$. A vector $(\vv,T, q)$ is a \emph{strong} solution to \eqref{eq2.1.7b}--\eqref{eq2.1.15b}  if 
\begin{alignat*}{2}
&(\vv,T, q)\in C([0,\T];V)\cap L^2(0,\T;W),\\
&\de_t (\vv,T, q)\in  L^2(0,\T;H),
\end{alignat*}
and, for almost every $t\in [0,t_1]$ and every $(\tilde\vv, \widetilde T,\tilde q)\in V$,
\begin{align}
&\la\de_t\vv, \tilde\vv\ra +a_{\vv}(\vv,\tilde\vv)+ b(\vv,\vv,\tilde\vv)+e_p(T,\tilde\vv)+e_c(\vv,\tilde\vv)=(\boldsymbol S_{\vv},\tilde\vv),\\
&\la\de_tT, \widetilde T\ra +a_{T}(T,\widetilde T)+ b(\vv,T,\widetilde T)=m_T(\omega,T,\widetilde T)
+\frac{L}{c_p}( D(\omega, T,h_q),\widetilde T)+(S_{T},\widetilde T),\\
&\la\de_t q, \tilde q\ra +a_{q}(q,\tilde q)+ b(\vv,q,\tilde q)
=-(D(\omega, T,h_q),\tilde q)+(S_{q},\tilde q),
\end{align}
for some $h_q\in L^\infty(\M\times (0,\T))$ which satisfies the variational inequality
\begin{equation}\label{newvar}
([\tilde q-q_s]^+,1)- ([q-q_s]^+,1)\geq  ( h_q, \tilde q-q), \qquad \text{a.e. } t\in [0,\T], \quad \forall\, \tilde q\in H^1.
\end{equation}
\end{definition}

The only difference between the definitions of quasi-strong and 
strong solutions lies in the regularity required for the triplet $(\vv,T,q)$. 

\subsection{Local existence of strong solutions}
From Corollary~\ref{cor3.2}, given initial data  $\vv_0\in \V$, $T_0\in H^1$, and $q_0\in H^1$, we deduce that a quasi-strong solution exists, at least locally in time. The velocity field $\vv$
already has the regularity required to be a strong solution. Our aim is now to improve the regularity on $T$ and $q$ on the same small time-interval of our local quasi-strong solution. We now prove the following theorem.

\begin{theorem}\label{thm5.1}
Let there be given  $\vv_0\in \V$ and $T_0,q_0\in H^1$. 
Assume $(\boldsymbol S_{\vv},S_T,S_q)\in L^2(0,\T;H)$, $q_s\in(0,1)$,
and $\T >0$.  Then there exists $t_*>0$ ($t_*\leq\T$) and  
a \emph{strong} solution to \eqref{eq:vTq1}--\eqref{eq:vTq3} such that 
\begin{alignat*}{2}
&(\vv,T,q)\in C([0,t_*);V)\cap L^2(0,t_*;W),\\
&\de_t(\vv,T,q)\in  L^2(0,t_*;H).
\end{alignat*}
\end{theorem}

\begin{proof}
Let $t_*>0$ be the time of existence of a local quasi-strong solution. We start with improving the regularity on $T$, showing that
$$
T\in L^\infty(0,t_*;H^1)\cap L^2(0,t_*;H^2).
$$
Testing formally equation \eqref{eq:vTq2} with $A_T T$, assuming that $A_T T\in L^2(0,\T; (H^1)')$, we obtain the differential equation
\begin{align*}
\frac12\frac{\dd}{\dd t}    |A_T^{1/2}T|^2 +|A_TT|^2=&- b(\vv,T,A_TT)+m_T(\omega,T,A_TT)\\
&+\frac{L}{c_p}(  D(\omega, T,h_q),A_TT)+(S_{T},A_TT).
\end{align*}
We now estimate the terms on the right. Thanks to \eqref{eq2.3.9} written only for $b$, we have
\begin{align*}
|b(\vv,T, A_TT) |&\leq c \|\vv\|^{1/2}\|\vv\|_{H^2}^{1/2} \| T \|^{1/2} \|T\|_{H^2}^{1/2} |A_TT|\\
&\leq c \|\vv\|^{1/2}\|\vv\|_{H^2}^{1/2} \| T \|^{1/2} |A_TT|^{3/2}\\
&\leq \frac18|A_TT|^2+c \|\vv\|^{2}\|\vv\|_{H^2}^{2} \| T \|^{2} 
\end{align*}
Moreover,
\begin{align*}
m_T(\omega,T,A_T T)&\leq c\|\omega\|_{L^4}\|T\|_{L^4}|A_TT|\\
&\leq c \|\omega\|\|T\||A_TT|\\
&\leq  \frac18|A_TT|^2 +c\|\vv\|_{H^2}^2\|T\|^2.
\end{align*}
The third term is estimated as
\begin{align*}
\frac{L}{c_p}(  D(\omega, T,h_q),A_TT)&\leq c|\omega^-||A_TT|\\
&\leq  \frac18|A_TT|^2 +c\|\vv\|^2,
\end{align*}
and, lastly,
$$
|(S_{T},A_TT)|\leq \frac18|A_TT|^2 +c|S_T|^2.
$$
Therefore, by equivalence of the norms $|A_T^{1/2}T|$ and $\|T\|$, we obtain
$$
\frac{\dd}{\dd t} |A_T^{1/2}T|^2 +|A_TT|^2\leq c \big(\|\vv\|^{2}\|\vv\|_{H^2}^{2}+\|\vv\|_{H^2}^2\big)|A_T^{1/2}T|^2+c\big(\|\vv\|^2+|S_T|^2\big).
$$
Note that since 
$$
\vv\in C([0,t_*);\V)\cap L^2(0,t_*;H^2)
$$
and
$$
S_T\in L^2(0,t_*;L^2),
$$
we see that 
$$
\|\vv\|^{2}\|\vv\|_{H^2}^{2}+\|\vv\|_{H^2}^2\in L^1(0,t_*), \qquad \|\vv\|^2+|S_T|^2\in L^1(0,t_*).
$$
Therefore, the claim that
$$
T\in L^\infty(0,t_*;H^1)\cap L^2(0,t_*;H^2)
$$
may be deduced from an application of the Gronwall lemma and the implementation of a Galerkin method. Once this is settled, the fact that
$$
\de_t T\in L^2(0,t_*;L^2)
$$
is deduced directly from equation \eqref{eq:vTq2}. This implies, in particular, that 
$$
T\in C([0,t_*);H^1)\cap L^2(0,t_*;H^2),
$$
concluding the proof. The regularity of $q$ can be established in the same way, as the $q$-equation
involves the same terms except for the trilinear form $m_T$. Theorem~\ref{thm5.1} is then proven.\end{proof}

\subsection{Global existence of strong solutions}
The existence of global strong solutions for the full primitive equations with saturation \eqref{eq1.1.1b}--\eqref{eq1.1.6b} with initial and boundary conditions \eqref{eq1.7}--\eqref{eq1.10} (or the reformulated version \eqref{eq:vTq1}--\eqref{eq:vTq3}) is a direct consequence of Theorem~\ref{thm6.1.1} from the Appendix and Theorem~\ref{thm5.1}.

\begin{theorem}\label{thm6.2.1}
Let there be given  $\vv_0\in \V$ and $T_0,q_0\in H^1$. 
Assume $(\boldsymbol S_{\vv},S_T,S_q)\in L^2(0,\T;H)$, $q_s\in(0,1)$ 
and $\T >0$.  Then there exists a global \emph{strong} solution to \eqref{eq:vTq1}--\eqref{eq:vTq3} such that 
\begin{alignat*}{2}
&(\vv,T,q)\in C([0,\T]; \V)\cap L^2(0,\T; W),\\
&\de_t(\vv,T,q)\in  L^2(0,\T;H).
\end{alignat*}
\end{theorem}

\begin{proof}
	From Theorem~\ref{thm5.1}, we already know the existence of a local strong solutions in some maximum time interval $[0, t_*)$ ($t_*\leq \T$). Therefore, in order to find the global strong solution, it is enough to show that the uniform bound $\norm{ (\bs v, T, q)(\cdot, t) } \leq \widetilde M$ independent of $t\in[0, t_*)$, which implies that no blow-up  can occur at the time $t=t_*$.
	
	Let us first rewrite the original velocity equation \eqref{eq2.1.7b} as
	\begin{equation}\label{eq6.2.3}
		\frac{\de \vv}{\de t} + \vv\cdot \nabla\vv + \omega \frac{\de \vv}{\de p}  + \nabla \Phi_s + \A_{\vv} \vv = \boldsymbol S_{\vv}-f\bs k\times\vv-\nabla\int_p^{p_1}\frac{R}{p'}T\dd p' =:\overline{\bs S}_{\vv}.\\
	\end{equation}
	Then by the $L^2$-uniform estimate \eqref{eq4.4.66} in Corollary~\ref{cor3.2}, we deduce 
	$$
	\norm{ \overline{\bs S}_{\vv} }_{L^2(0, t_*; L^2)} \leq \widetilde M,
	$$
	for some $\widetilde M>0$ independent of $t_*$. Therefore, applying Theorem~\ref{thm6.1.1} to \eqref{eq6.2.3} with $\bs S_{\vv}=\overline{\bs S}_{\vv}$, we obtain 
	\begin{equation}\label{eq6.2.5}
		\norm{\vv(t)}+ \norm{\vv}_{L^2(0,t_*; H^2)} \leq \widetilde M,\quad\quad\text{independent of }t_*.
	\end{equation}
	With the uniform estimate \eqref{eq6.2.5} for $\vv$ at hand, proceeding exactly as in the proof of Theorem~\ref{thm5.1} to seek the uniform $H^1$-estimate for $T$ and $q$, we are able to obtain the uniform bound for $\norm{ (T, q)(\cdot, t) }_V$. In conclusion, we have
	$$
	\sup_{t\in[0,t_*)}\norm{ (\bs v, T, q)( t) }+\norm{(\vv, T, q)}_{L^2(0,t_*; H^2)}  \leq \widetilde M, 
	$$
	for some $\widetilde M>0$ independent of $t_*$. We thus completed the proof of the existence of global strong solutions to \eqref{eq:vTq1}--\eqref{eq:vTq3}, that is, the proof of Theorem~\ref{thm6.2.1}.
\end{proof}

\section{Uniqueness of strong and quasi-strong solutions}\label{sec:uniq}
\noindent Here we prove uniqueness and continuous dependence results for the quasi-strong solutions 
to \eqref{eq:vTq1}--\eqref{eq:vTq3}. As a straightforward consequence, strong solutions turn out to be unique as well. 
The proof combines many ideas. The velocity equations are treated as in
\cites{CT06,GMR,TZ04}, while the temperature equation has to be substituted by the moist static
energy equation in order to exploit a certain cancellation property, as in \cites{CFTT}. Finally, tools from 
monotone operator theory and variational inequalities \cites{B,ET76,KS,R}  turn out to be useful to handle 
the specific humidity equation, again  following along the lines in \cite{CFTT}. 
The following is the main result 
of this section.

\begin{theorem}\label{thm:unieq}
Assume that $(\vv_1,T_1,q_1)$ and $(\vv_2,T_2,q_2)$ 
are two (strong or) quasi-strong solutions to \eqref{eq:vTq1}--\eqref{eq:vTq3} on $[0,\T]$, with the nonlinear function $F$ replaced by
its positive part $F^+$, and with initial data $\vv_i^0\in \V$, an $T_i^0,q_i^0\in L^2$, 
for $i=1,2$. Then, there exists positive
constants $c$ and $c_0=c_0(\|\vv_i^0\|,|T_i^0|,|q_i^0|)$ such that
$$
\sup_{t\in[0,\T]} \|(\vv_1,T_1,q_1)-(\vv_2,T_2,q_2)\|_{\V\times L^2\times L^2}^2(t)
\leq c\,\e^{c_0\T}\|(\vv_1^0,T_1^0,q_1^0)-(\vv_2^0,T_2^0,q_2^0)\|^2_{\V\times L^2\times L^2}.
$$
In particular, there exists a unique  strong solution to \eqref{eq:vTq1}--\eqref{eq:vTq3}.
\end{theorem}

\begin{remark}
The replacement of $F$ with its positive part $F^+$ plays here an essential role. It is linked with
positivity and $L^\infty$ bounds on the temperature, which will be discussed in the subsequent Section
\ref{sec:maxprinc}. Note that Theorem \ref{thm:unieq} applies as well when we keep $F$ instead of replacing it 
by $F^+$, as long as the temperature remains positive and below the bound 
$$
T\leq \frac{L R}{c_pR_v}\simeq 1548K,
$$
which is far higher than any temperature on earth.
\end{remark}

\subsection{The moist static energy}
In order to prove the uniqueness of solutions in Theorem~\ref{thm:unieq}, we introduce as in \cite{CFTT} the so-called \emph{moist static energy} function
$$
e=c_p T+ Lq,
$$
which is easily seen to satisfy the equation
\begin{equation}\label{eq:ephy}
\frac{\de e}{\de t}+\vv \cdot \nabla e +\omega \frac{\de e}{\de p}+\A_T e=L(\A_T -\A_q)q+\frac{R}{c_pp}\omega (e-Lq),
\end{equation}
along with the boundary conditions
\begin{equation}\label{BC1}
\begin{aligned}
&\text{on }\Gamma_i: \frac{\de e}{\de p}=-\frac{\alpha_T}{\nu_T}e+L\left( \frac{\alpha_T}{\nu_T}-\frac{\alpha_q}{\nu_q}\right)q,\\
&\text{on }\Gamma_u:  \frac{\de e}{\de p}=0,\\
&\text{on }\Gamma_\ell:  \frac{\de e}{\de \boldsymbol n}=0.
\end{aligned}
\end{equation}
Also, the initial condition now reads
\begin{equation}\label{eq:init3}
e(x,y,p,0)=e_0(x,y,p),
\end{equation}
where  $e_0=c_p T_0+Lq_0$. It is clear that proving uniqueness of a quasi-strong solution $(\vv,T,q)$ is equivalent to showing uniqueness of a quasi-strong solution $(\vv,e,q)$. Although consideration of the $(\vv,e,q)$-system introduces some coupling in the linear part
of \eqref{eq:ephy} and in the boundary conditions on $\Gamma_i$, these terms can be handled and  we gain the advantage of eliminating the
non-Lipschitz term induced by the multivalued function $H(q-q_s)$. 

\subsection{Proof of Theorem~\ref{thm:unieq}}
In what follows, the letter $c$ will refer to a \emph{generic} positive constant, which may be calculated in terms
of the physical parameters of the problem. Also, $\delta\in (0,1)$ will be a sufficiently small \emph{fixed} constant.
Let $(\vv_1,e_1,q_1)$ and $(\vv_2,e_2,q_2)$ be two global quasi-strong solutions with initial data
$(\vv_1^0,e_1^0,q_1^0)$ and $(\vv_2^0,e_2^0,q_2^0)$, respectively.
The differences
$$
\hv=\vv_1-\vv_2,\qquad \he=e_1-e_2,\qquad \hq=q_1-q_2,
$$
satisfy 
$$
\begin{aligned}
&\de_t \hv + (\vv_1\cdot\nabla)\hv +(\hv\cdot\nabla) \vv_2+\omega_1 \de_p\hv +\ho \de_p\vv_2+ f\kk\times \hv+\nabla \widehat\Phi_s\\
&\qquad+\nabla\int_p^{p_1} \frac{R}{c_pp^\prime} \big[\he-L\hq\big] \dd p^\prime +\A_{\vv}\hv=0,\\
&\de_t \he+ \vv_1\cdot\nabla\he +\hv\cdot\nabla e_2+\omega_1 \de_p\he +\ho \de_pe_2+\A_T\he= L(\A_T -\A_q)\hq\\
&\qquad+\frac{R}{c_pp}\omega_1 (\he-L\hq) +\frac{R}{c_pp}\ho(e_2-Lq_2) ,\\
&\de_t \hq  + \vv_1\cdot\nabla\hq +\hv\cdot\nabla q_2+\omega_1 \de_p\hq +\ho \de_pq_2+\A_q\hq 
+\frac1p\omega_1^-F^+(T_2) (h_{q_1}-h_{q_2}) \\
&\qquad=-\frac1p\omega_1^-h_{q_1} (F^+(T_1)-F^+(T_2))-\frac1ph_{q_2}F^+(T_2)(\omega_1^--\omega_2^-) .
\end{aligned}
$$
We now proceed to test each equation in $L^2$ by $A_{\vv}\hv$, $\he$ and $\hq$, respectively, taking into account the 
weak formulation of the equations given by \eqref{eq:vTq1}--\eqref{eq:vTq3}, appropriately reformulated for $e$ instead of
$T$. For $\hv$, 
we obtain the energy equation
\begin{equation}\label{vest}
\begin{aligned}
\frac12\frac{\dd }{\dd t}|A_{\vv}^{1/2}\hv|^2+|A_{\vv}\hv|^2=&-\left(\nabla\int_p^{p_1} \frac{R}{c_pp^\prime} \big[\he-L\hq\big] \dd p^\prime,A_{\vv}\hv\right)
-b(\hv, \vv_2,A_{\vv}\hv)\\
&-b(\vv_1, \hv,A_{\vv}\hv)-(f\kk\times \hv,A_{\vv}\hv).
\end{aligned}
\end{equation}
We estimate the right-hand side term by term. For the last term, we use the Poincar\'e inequality for the estimate
\begin{equation}\label{vest0}
|(f\kk\times \hv,A_{\vv}\hv)|\leq c\|\hv\|^2+ \frac{1}{24} |A_{\vv}\hv|^2.
\end{equation}
For the first one, as in \eqref{eq2.3.6} we obtain
\begin{equation}\label{vest1}
\left|\left(\nabla\int_p^{p_1} \frac{R}{c_pp^\prime} \big[\he-L\hq\big] \dd p^\prime,A_{\vv}\hv\right)\right|
\leq c(\|\he\|+\|\hq\|)|A_{\vv}\hv|\leq
\frac{1}{24} |A_{\vv}\hv|^2+\frac{\kappa_T}{4\delta}\|\he\|^2+\frac{\kappa_q}{4\delta^2}\|\hq\|^2.
\end{equation}
Now, the trilinear terms are more difficult and require some care. Note that
$$
b(\hv, \vv_2,A_{\vv}\hv)=((\hv\cdot\nabla) \vv_2, A_{\vv}\hv)+(\ho \de_p \vv_2, A_{\vv}\hv),
$$
so that the first part can be estimated in a fairly classical way as
$$
|((\hv\cdot\nabla) \vv_2, A_{\vv}\hv)|\leq \|\hv\|_{L^6}\|\nabla \vv_2\|_{L^3}||A_{\vv}\hv|\leq c\|\hv\| \|\vv_2\|_{H^2}|A_{\vv}\hv|
\leq \frac{1}{24}|A_{\vv}\hv|^2+ c\|\vv_2\|_{H^2}^2\|\hv\|^2.
$$
For the second part, we use an anisotropic estimate. We obtain
$$
|(\ho \de_p \vv_2, A_{\vv}\hv)|\leq c\int_{\M'} \|\ho\|_{L^\infty_p} \|\vv_2\|_{L_p^2}\|\|A_{\vv}\hv\|_{L^2_p}\dd \M',
$$
where the subscript $p$ in the above norms indicates that we have only integrated in the $p$-direction. From the definition
of $\ho$,  we have that
$$
\|\ho\|_{L^\infty_p}\leq c\|\nabla\cdot\hv\|_{L^2_p},
$$
so that a further use of the H\"older inequality in the $x,y$ direction entails
\begin{align*}
|(\ho \de_p \vv_2, A_{\vv}\hv)|&\leq c\|\ho\|_{L^4_{x,y}L^\infty_p}\|\de_p\vv_2\|_{L^4_{x,y}L_p^2}\|\|A_{\vv}\hv\|_{L^2_{x,y}L^2_p}\\
&\leq c \|\nabla \cdot\hv\|_{L^4_{x,y}L^2_p} \|\de_p\vv_2\|_{L^4_{x,y}L_p^2}\|\|A_{\vv}\hv\|_{L^2_{x,y}L^2_p}\\
&\leq c\|\hv\|^{1/2}\|\vv_2\|^{1/2}\|\vv_2\|_{H^2}^{1/2}|A_{\vv}\hv|^{3/2}\\
&\leq \frac{1}{24}  |A_{\vv}\hv|^2 + c \|\vv_2\|^{2}\|\vv_2\|_{H^2}^2 \|\hv\|^2,
\end{align*}
where we used that for $\varphi\in H^1$ there holds
$$
\|\varphi\|_{L^4_{x,y}L_p^2}\leq c \|\varphi\|_{L^2}^{1/2}\|\varphi\|_{H^1}^{1/2}.
$$
In conclusion, we find that
\begin{equation}\label{vest2}
|b(\hv, \vv_2,A_{\vv}\hv)|\leq \frac{1}{12}|A_{\vv}\hv|^2+ c\big(1+ \|\vv_2\|^{2}\big) \|\vv_2\|_{H^2}^2\|\hv\|^2.
\end{equation}
We argue similarly for the third term. Since 
$$
b(\vv_1, \hv,A_{\vv}\hv)=((\vv_1\cdot\nabla) \hv, A_{\vv}\hv)+(\omega_1 \de_p \hv, A_{\vv}\hv),
$$
we deduce that
$$
|((\vv_1\cdot\nabla) \hv, A_{\vv}\hv)|\leq \|\vv_1\|_{L^6} \|\nabla\hv\|_{L^3}|A_{\vv}\hv|\leq c \|\vv_1\| \|\hv\|^{1/2}|A_{\vv}\hv|^{3/2}
\leq \frac{1}{24} |A_{\vv}\hv|^2+c   \|\vv_1\|^4 \|\hv\|^2,
$$
and
\begin{align*}
|(\omega_1 \de_p \hv, A_{\vv}\hv)|&\leq  c \|\omega_1\|_{L^4_{x,y}L^\infty_p} \|\de_p\hv\|_{L^4_{x,y}L^2_p}|A_{\vv}\hv| \\
&\leq  c \|\nabla\cdot \vv_1\|_{L^4_{x,y}L^2_p} \|\de_p\hv\|_{L^4_{x,y}L^2_p}|A_{\vv}\hv| \\
&\leq c \|\vv_1\|^{1/2}\|\vv_1\|_{H^2}^{1/2}\|\hv\|^{1/2}|A_{\vv}\hv|^{3/2}\\
&\leq \frac{1}{24}|A_{\vv}\hv|^2 + c \|\vv_1\|^2\|\vv_1\|_{H^2}^2\|\hv\|^2 ,
\end{align*}
implying
\begin{equation}\label{vest3}
|b(\vv_1, \hv,A_{\vv}\hv)|\leq \frac{1}{12}|A_{\vv}\hv|^2 + c \big(\|\vv_1\|^2+\|\vv_1\|_{H^2}^2\big)\|\vv_1\|^2 \|\hv\|^2.
\end{equation}
Hence, in light of \eqref{vest0}--\eqref{vest3}, we derive from \eqref{vest} the differential inequality
\begin{equation}\label{vestineq}
\begin{split}
	\frac{\dd }{\dd t}\|\hv\|^2+|A_{\vv}\hv|^2
	\leq &\frac{\kappa_T}{2\delta}\|\he\|^2+\frac{\kappa_q}{2\delta^2}\|\hq\|^2\\
	&+ c\big[\|\vv_1\|^4+\|\vv_1\|^2\|\vv_1\|_{H^2}^2+ \|\vv_2\|^2_{H^2}+ 
\|\vv_2\|^2\|\vv_2\|^2_{H^2}\big]\|\hv\|^2.
\end{split}
\end{equation}
We now turn our attention to the moist static energy equation. Applying the same technique, we have 
\begin{equation}\label{eest}
\begin{aligned}
\frac12\frac{\dd}{\dd t}|\he|^2+a_T(\he,\he)=&-b(\hv,e_2,\he)+
\frac{R}{c_p}(\frac1p\omega_1(\he-L\hq),\he)\\
&+\frac{R}{c_p}(\frac1p\ho(e_2-Lq_2),\he)
+L(a_T(\hq,\he)-a_q(\hq,\he)).
\end{aligned}
\end{equation}
Thanks to the orthogonality property of the trilinear form, we have 
$$
-b(\hv,e_2,\he)=b(\hv,\he,e_2)=(\hv\cdot \nabla \he,e_2)+(\ho\de_p\he,e_2).
$$
Therefore, as above,
\begin{equation}\label{eest1}
|(\hv\cdot \nabla \he,e_2)|\leq c\|\hv\|_{L^6}\|\he\|\|e_2\|_{L^3}\leq c\|\hv\|\|\he\|\|e_2\|
\leq\frac{\kappa_{T}}{10}\|\he\|^2+c\|e_2\|^2\|\hv\|^2,
\end{equation}
and by anisotropic estimates,
\begin{equation}\label{eest2}
\begin{aligned}
|(\ho \de_p\he,e_2)|&\leq c\|\ho\|_{L^4_{x,y}L^\infty_p}\|\de_p\he\|_{L^2_{x,y}L^2_p}\|e_2\|_{L^4_{x,y}L^2_p}\\
&\leq c\|\nabla\cdot\hv\|_{L^4_{x,y}L^2_p}\|\de_p\he\|_{L^2_{x,y}L^2_p}\|e_2\|_{L^4_{x,y}L^2_p}\\
&\leq c\|\nabla\cdot\hv\|_{L^4_{x,y}L^2_p}|\de_p\he|\|e_2\|_{L^4_{x,y}L^2_p}\\
&\leq c \|\hv\|^{1/2} |A_{\vv}\hv|^{1/2} |e_2|^{1/2}\|e_2\|^{1/2}\|\he\|\\
&\leq \frac{\kappa_T}{10}\|\he\|^2+ \frac{\delta}{8}|A_{\vv}\hv|^2+c|e_2|^2\|e_2\|^2\|\hv\|^2
\end{aligned}
\end{equation}
Thanks to the continuity of the bilinear forms $a_T$ and $a_q$, we are able to estimate the last term as
\begin{equation}\label{eest3}
|L(a_T(\hq,\he)-a_q(\hq,\he))|\leq c\|\he\|\|\hq\|\leq \frac{\kappa_{T}}{10} \|\he\|^2+\frac{\kappa_{q}}{8\delta} \|\hq\|^2.
\end{equation}
Regarding the intermediate terms, we have
\begin{equation}\label{eest4}
\begin{aligned}
\left|\frac{R}{c_p}\left(\frac1p\omega_1(\he-L\hq),\he\right)\right|&\leq c|\omega_1|\big(\|\he\|^2_{L^4}+\|\he\|_{L^4}\|\hq\|_{L^4}\big)\\
&\leq c\|\vv_1\|\big(|\he|^{1/2}\|\he\|^{3/2}+|\he|^{1/4}\|\he\|^{3/4}|\hq|^{1/4}\|\hq\|^{3/4}\big)\\
&\leq \frac{\kappa_{T}}{10}\|\he\|^2+c\|\vv_1\|^4(|\he|^2+|\hq|^2)+\frac{\kappa_{q}}{8\delta}\|\hq\|^2,
\end{aligned}
\end{equation}
where we took advantage of the Sobolev embedding $H^{3/4}\subset L^4$ and interpolation inequalities. In a similar manner,
\begin{equation}\label{eest5}
\begin{aligned}
\left|\frac{R}{c_p}\left(\frac1p\ho(e_2-Lq_2),\he\right)\right|&\leq  c\|\ho\|_{L^6} (|e_2|+|q_2\|)\|\he\|_{L^3} \\
&\leq c|A_{\vv}\hv| (|e_2|+|q_2|)|\he|^{1/2}\|\he\|^{1/2} \\
&\leq  \frac{\delta}{8}|A_{\vv}\hv|^2+\frac{\kappa_{T}}{10}\|\he\|^2+c(|e_2|^4+|q_2|^4)|\he|^2. 
\end{aligned}
\end{equation}
Therefore, in view of \eqref{eest1}--\eqref{eest5}, the energy equation \eqref{eest} becomes
\begin{equation}\label{eestineq}
\begin{aligned}
\frac{\dd}{\dd t}|\he|^2+\kappa_{T}\|\he\|^2&
\leq \frac{\delta}{2}|A_{\vv}\hv|^2 +\frac{\kappa_{q}}{2\delta}\|\hq\|^2+c \big[1+|e_2|^2\big]\|e_2\|^2\|\hv\|^2 +c\|\vv_1\|^4 \big(|\he|^2+|\hq|^2\big).
\end{aligned}
\end{equation}
It remains to deal with the specific humidity equation. The corresponding energy equality reads
\begin{equation}\label{qest}
\begin{aligned}
&\frac12\frac{\dd}{\dd t}|\hq|^2+a_q(\hq,\hq)+\left(\frac1p\omega_1^-F^+(T_2)(h_{q_1}-h_{q_2}),\hq \right)=-b(\hv,q_2,\hq)\\
&\qquad
-\left(\frac1p\omega_1^-h_{q_1}(F^+(T_1)-F^+(T_2)),\hq\right)-(\frac1ph_{q_2}F^+(T_2)(\omega_1^--\omega_2^-),\hq).
\end{aligned}
\end{equation}
As a consequence of the monotonicity of the multivalued map $q\mapsto H(q-q_s)$, we find
\begin{equation}\label{qest1}
\left(\frac1p\omega_1^-F^+(T_2)(h_{q_1}-h_{q_2}),\hq \right)\geq 0
\end{equation}
Also, arguing as in \eqref{eest1}--\eqref{eest2}, we infer that
\begin{equation}\label{qest2}
\begin{aligned}
|b(\hv,q_2,\hq)|\leq\frac{\delta^2}{4}|A_{\vv}\hv|^2+\frac{\kappa_{q}}{6}\|\hq\|^2+c(1+|q_2|^2) \|q_2\|^2\|\hv\|^2.
\end{aligned}
\end{equation}
Moreover, since $\|h_{q_1}\|_{L^\infty}\leq 1$ and $F^+$ is globally Lipschitz-continuous, we can write
\begin{equation}\label{qest3}
\begin{aligned}
\left|\left(\frac1p\omega_1^-h_{q_1}(F^+(T_1)-F^+(T_2)),\hq\right)\right|&\leq c|\omega_1^-|\| F(T_1)-F(T_2)\|_{L^6}\|\hq\|_{L^3}\\
&\leq c \|\vv_1\| \|\hT\|_{L^6}|\hq|^{1/2}\|\hq\|^{1/2}\\
&\leq c \|\vv_1\| (\|\he\|+\|\hq\|)|\hq|^{1/2}\|\hq\|^{1/2}\\
&\leq \frac{\delta \kappa_{T}}{4} \|\he\|^2+\frac{\kappa_{q}}{6}\|\hq\|^2+c\|\vv_1\|^4|\hq|^2.
\end{aligned}
\end{equation}
In a simpler way, recalling
that $F$ is a bounded function, we obtain
\begin{equation}\label{qest4}
\left|\left(\frac1ph_{q_2}F^+(T_2)(\omega_1^--\omega_2^-),\hq\right)\right|\leq c|\ho||\hq|\leq c\|\hv\|^2+\frac{\kappa_{q}}{6} \|\hq\|^2.
\end{equation}
Collecting \eqref{qest1}--\eqref{qest4}, the equation \eqref{qest} turns into
\begin{equation}\label{qestineq}
\frac{\dd}{\dd t}|\hq|^2+\kappa_{q}\|\hq\|^2\leq \frac{\delta^2}{2}|A_{\vv}\hv|^2 
+\frac{\delta\kappa_{T}}{2}\|\he\|^2+c \|\vv_1\|^4|\hq|^2+ c(1+\|q_2\|^2+|q_2|^2\|q_2\|^2)\|\hv\|^2.
\end{equation}
We now use the fact that the norm $|A_{\vv}^{1/2}\hv|$ is equivalent to the norm $\|\hv\|$ and put together \eqref{vestineq}, \eqref{eestineq}, and \eqref{qestineq} in the following way. The energy functional
$$
\Psi(t)=\delta^2|A_{\vv}^{1/2}\hv(t)|^2+\delta|\he(t)|^2+|\hq(t)|^2
$$
satisfies the differential inequality
$$
\Psi^\prime(t) \leq g(t) \Psi(t), \qquad t\in(0,\T],
$$
where 
\begin{align*}
g(t)=c\big[&1+\|\vv_2\|_{H^2}^2+ \|\vv_1\|^2\|\vv_1\|_{H^2}^2+ \|\vv_2\|^2\|\vv_2\|_{H^2}^2 \\ 
&+\|\vv_1\|^4+\|e_2\|^2+|e_2|^2\|e_2\|^2
+\|q_2\|^2+|q_2|^2\|q_2\|^2\big].
\end{align*}
The fact that  $(\vv_1,e_1,q_1)$ and $(\vv_2,e_2,q_2)$ are quasi-strong solutions ensures that 
$g\in L^1(0,\T)$. We can therefore apply the standard Gronwall inequality, to obtain the continuous 
dependence estimate
$$
\Psi(t)\leq \Psi(0)\exp\left(\int_0^{\T}g(t)\dd t\right),
$$
which is stated explicitly in the statement of Theorem~\ref{thm:unieq}, and whose proof is now achieved.

\begin{remark}
The only step in which the replacement of $F$ with its positive part $F^+$ has been used is in \eqref{qest1},
to exploit monotonicity properties of the set-valued map $H(q-q_s)$. This appears unavoidable at the moment,
and it allows to circumvent the difficulty of dealing with non-Lipschitz nonlinearities, by using instead monotonicity 
arguments.
\end{remark}

\section{Maximum principles}\label{sec:maxprinc}

\noindent In order to prove uniqueness of solutions, we  replaced
the nonlinear function $F$ by $F^+$. 
This amounts to requiring that the temperature distribution $T$ satisfies the bounds given by \eqref{eq:nonbound}, namely
\begin{equation}\label{eq:nonbound2}
0\leq T\leq \frac{L R}{c_pR_v}\simeq 1548K, \qquad \text{a.e. in } \M\times [0,\T].
\end{equation}
Note in addition that the physical model \eqref{eq1.1.1}--\eqref{eq1.1.6} is probably not valid anymore if
$T>1548K$, for possibile utilization in extraterrestrial atmospheres. 

In \cite{CFTT}, physical bounds of the form \eqref{eq:nonbound2} were derived for both the temperature and the specific
humidity. Here, the picture is quite similar, despite the fact that the velocity vector field $\vv$ is not a given datum
anymore, but part of the unknowns. In particular, one can repeat word for word the arguments in 
\cite[Proposition 4.1]{CFTT}
to obtain a positivity result.

\begin{proposition}\label{prop:posit}
With the hypotheses in Theorem~\ref{thm6.2.1}, 
suppose $T_0,q_0$ and $S_T,S_q$ are positive functions. 
Then we have $T(\cdot,t),q(\cdot,t)\geq 0$  almost everywhere in $\M$ and for every  
$t\in [0,\T]$.
\end{proposition}

The main difference compared to \cite{CFTT} concerns the upper bounds. Specifically, a key assumption there was 
that $\omega\in L^\infty(\M\times(0,\T))$, which is not compatible with our estimates above on strong solutions. Here, we will instead
exploit a technique reminiscent of the famous nonlinear iteration of De Giorgi \cite{DG} in order to prove the following
$L^\infty$ bound on the temperature.

\begin{proposition}\label{prop:Tbound}
With the hypotheses in Theorem~\ref{thm6.2.1}, we
suppose that $T_0\in L^\infty$ is positive and we further assume that 
$S_T\in L^2(0,\T;H^1)$.
Then there exists  a positive constant
$M_1$ such that
\begin{equation}\label{eq:Tqbound}
0\leq T\leq M_1, \qquad \text{a.e. in } \M\times[0,\T].
\end{equation}
Specifically, we have
$$
M_1=M_1(\T,\|\vv_0\|,\|T_0\|, \|T_0\|_{L^\infty},\|q_0\|,\|S_T\|_{L^2(0,\T;H^1)},\norm{(\bs S_{\vv}, S_{T}, S_{q})}_{L^2(0,\T; H)}).
$$
\end{proposition}
The only extra requirement is $S_T\in L^2(0,\T;H^1)$. The above proposition will be proven in the 
subsequent Section \ref{sub:degiorgi}. Moreover, as the equation for $q$ is similar to the equation for $T$,
an analogous result holds for the specific humidity.

\begin{proposition}\label{prop:qbound}
With the hypotheses in Theorem~\ref{thm6.2.1}, assume that
$q_0\in L^\infty$ is positive and
$S_q\in L^2(0,\T;H^1)$. Then
there exist  a positive constant
$M_2$ such that
\begin{equation}\label{eq:Tqbound2}
0\leq q\leq M_2, \qquad \text{a.e. in } \M\times[0,\T].
\end{equation}
Specifically, we have
$$
M_2=M_2(\T,\|\vv_0\|,\|T_0\|,\|q_0\|,\|q_0\|_{L^\infty},\|S_q\|_{L^2(0,\T;H^1)},\norm{(\bs S_{\vv}, S_{T}, S_{q})}_{L^2(0,\T; H)}).
$$
\end{proposition}

The proof of the above proposition will be left to the reader, as it is shorter and simpler than that of
Proposition \ref{prop:Tbound}.

\begin{remark}\label{rmk:dissp}
The main drawback of such bounds is the dependence on the final time $\T$. However, this is not caused
by a flaw in the proof, but rather from the non-uniformity of the basic $L^2$ energy estimate \eqref{eq4.4.66}.
This seems to be an intrinsic feature of the model, caused essentially by the presence of the anti-dissipative
term 
$$
-\frac{R}{c_p p}\omega T
$$
in the temperature equation.
\end{remark}

The methods of \cite{CFTT} work here as well for the specific humidity, without further assumptions on $\omega$. 
Complementary to Proposition \ref{prop:qbound} in the (physical) case of zero forcing, we also have the following:
\begin{proposition}\label{prop:qbound2}
With the hypotheses in Theorem~\ref{thm6.2.1}, assume that $q$ is a strong solution 
to \eqref{eq:vTq3} on $[0,\T]$, with the nonlinear function $F$ replaced by
its positive part $F^+$, and with positive initial data $q_0\in L^\infty$.
If $S_q=0$, then
\begin{equation}\label{eq:Tqbound3}
0\leq q\leq \|q_0\|_{L^\infty} \qquad \text{a.e. in } \M\times[0,\T].
\end{equation}
\end{proposition}

\begin{remark}
Property \eqref{eq:Tqbound3} is what one should expect from the physical 
considerations. Indeed, the specific humidity is the (dimensionless) 
ratio of the mass of water vapor to the total mass of the system.
One therefore expects that
$$
q\in [0,1], \qquad \text{a.e. in } \M\times [0,\T].
$$ 
Moreover, \eqref{eq:Tqbound3} also implies that if the initial datum is smaller than the saturation
concentration, then the under-saturated regime persists for all times.
If, as it should be (\cite[p.~163]{Hal71}), $q_s$ depends on $T$ (typically, 
$q_s=\frac{C}{p}\e^{-L/R_vT}$, as in \cite{Hal71}), then a decrease of $T$ produces a decrease of $q_s$,
thus increasing the chances of supersaturation. The mathematical theory of 
the case where $q_s$ depends on $T$ has been developed in \cite{BCT}.
\end{remark}

\begin{remark}
We observe once more that the truncation of $F$ is not required to prove the maximum principle
neither for the temperature nor for the specific humidity, as opposed to what was proved in
\cite{CFTT}, where the truncation of $F$ was assumed to prove the upper bound for $q$.  Indeed,
it was there crucial an estimate of the type
$$
-\la\frac1p  \omega^-h_q F ^+(T),[q-M_2]^+\ra\leq 0,
$$
when using the Stampacchia method on the specific humidity equation.
However, a bound in space and time (\emph{independent} of the final time $\T$), like in \eqref{eq:nonbound2} 
seems to be out of reach at the moment: the possibility of exponential growth of the $L^2$ norm
of the temperature (see \eqref{eq4.4.66}) caused by the anti-dissipative term mentioned in Remark \ref{rmk:dissp}
prevents us to prove uniform $L^\infty$ bounds as well (see \eqref{eq:iter3} and \eqref{eq:argentina}).
\end{remark}

\subsection{Upper bounds for the temperature via De Giorgi iterations}\label{sub:degiorgi}
We prove here Proposition \ref{prop:Tbound}.
We consider the weak formulation of the temperature equation \eqref{eq:vTq2}, and take the test
function $\widetilde T$ to be $T_\lambda=[T-\lambda]^+$, where $\lambda$ is any positive constant
such that
$$
\lambda\geq \|T_0\|_{L^\infty}.
$$
The corresponding energy equation reads
\begin{align*}
\frac12\ddt |T_\lambda|^2 + a_T(T_\lambda,T_\lambda)+ 
\lambda\alpha_T\int_{\Gamma_i} \left(\frac{gp_1}{R\overline T}\right)^2T_\lambda\,\dd\Gamma_i &=
m_T(\omega,T_\lambda,T_\lambda) +m_T(\omega,\lambda,T_\lambda)\\
&\quad+\frac{L}{c_p}( D(\omega, T,h_q),T_\lambda)+(S_{T},T_\lambda),
\end{align*}
where we took advantage of the orthogonality property of the trilinear form $b$, namely
$$
b(\vv,T,T_\lambda)=b(\vv,T-\lambda,T_\lambda)=b(\vv,T_\lambda,T_\lambda)=0.
$$
As a consequence, we have 
\begin{align*}
\ddt |T_\lambda|^2 +2 \kappa_T\|T_\lambda\|^2
 &\leq
c_1\int_\M |\omega||T_\lambda|^2 \dd\M+c_1\lambda\int_\M |\omega||T_\lambda| \dd\M\\
&\quad+c_2\int_\M |\omega||T_\lambda| \dd\M+2\int_\M |S_T||T_\lambda| \dd\M,
\end{align*}
where
$$
c_1=\frac{Rc_p}{p_0}, \qquad c_2=\frac{Lc_p}{p_0}C_F,\qquad (C_F>0 \text{ from \eqref{N2}}).
$$
Hence, integrating on $(0,t)$ and using the fact that $T_\lambda(0)=0$, we obtain
\begin{equation}\label{eq:iter1}
\begin{aligned}
\sup_{t\in[0,\T]}|T_\lambda(t)|^2 +2 \kappa_T\int_0^{\T}\|T_\lambda(t)\|^2\dd t
 &\leq
c_1\int_0^{\T}\int_\M |\omega||T_\lambda|^2 \dd\M\dd t+c_1\lambda\int_0^{\T}\int_\M |\omega||T_\lambda| \dd\M\dd t\\
&\quad+c_2\int_0^{\T}\int_\M |\omega||T_\lambda| \dd\M\dd t+2\int_0^{\T}\int_\M |S_T||T_\lambda| \dd\M\dd t.
\end{aligned}
\end{equation}
For $M\geq 2\|T_0\|_{L^\infty}$ to be fixed later, we define the increasing sequence of positive numbers
$$
\lambda_k=M(1-2^{-k}),\qquad k\geq 0,
$$
and set $T_k=T_{\lambda_k}$. Note that $T_k\leq T_{k-1}$ for every $k\geq 0$, and, moreover,
$$
T_{k-1}\geq 2^{-k}M \qquad \text{on } \{(x,y,p,t): T_k(x,y,p,t)>0\}.
$$
In turn,
\begin{equation}\label{eq:iter11}
\mathds{1}_{\{T_k>0\}}\leq \frac{2^k}{M}T_{k-1}, \quad \forall k\geq 1.
\end{equation}
Defining 
\begin{equation}\label{eq:iter0}
Q_k=\sup_{t\in[0,\T]}|T_k(t)|^2 +2 \kappa_T\int_0^{\T}\|T_k(t)\|^2\dd t,
\end{equation}
we infer from \eqref{eq:iter1} that, for every $k\geq 1$,
\begin{equation}\label{eq:iter2}
\begin{aligned}
Q_k &\leq
c_1\int_0^{\T}\int_\M |\omega||T_k|^2 \dd\M\dd t+c_1M\int_0^{\T}\int_\M |\omega||T_k| \dd\M\dd t\\
&\quad+c_2\int_0^{\T}\int_\M |\omega||T_k| \dd\M\dd t+2\int_0^{\T}\int_\M |S_T||T_k| \dd\M\dd t,
\end{aligned}
\end{equation}
where we used  $\|T_0\|_{L^\infty}\leq \lambda_k\leq M$  for each $k\geq 1$ from our choice of $M$. Moreover,
from the estimate \eqref{eq4.4.66}, which can now be extended up to $\T$ in view of the global existence result in Theorem
\ref{thm6.2.1}, we have
\begin{equation}\label{eq:iter3}
Q_0=\sup_{t\in[0,\T]}|T(t)|^2 +2 \kappa_T\int_0^{\T}\|T(t)\|^2\dd t\leq Ce^{C\T}\big(|(\vv_0,T_0,q_0)|^2 +\norm{(\bs S_{\vv}, S_{T}, S_{q})}_{L^2(0,\T; H)}^2\big).
\end{equation}
We now proceed with estimating the right hand side of \eqref{eq:iter2}. Note that it is enough to find upper bounds for the first two terms, as the third and fourth terms are very similar to the second one.

We have
\begin{align*}
\int_0^{\T}\int_\M |\omega||T_k|^2\dd\M\dd t&\leq \int_0^{\T}\|\omega\|_{L^6}\left(\int_\M |T_k|^{12/5}\dd\M\right)^{5/6}\dd t \\
&\leq
\|\omega\|_{L^2(0,\T;H^1)}\left(\int_0^{\T}\left(\int_\M |T_k|^{12/5}\dd\M\right)^{5/3}\dd t \right)^{1/2}.
\end{align*}
In view of \eqref{eq:iter11} and the fact that $T_k\leq T_{k-1}$, we have
\begin{align*}
\int_0^{\T}\left(\int_\M |T_k|^{12/5}\dd\M\right)^{5/3}\dd t
&\leq \int_0^{\T}\left(\int_\M |T_{k-1}|^{12/5}\mathds{1}_{\{T_k>0\}}^{2/5}\dd\M\right)^{5/3}\dd t  \\
&\leq \frac{2^{2k/3}}{M^{2/3}}\int_0^{\T}\left(\int_\M |T_{k-1}|^{14/5}\dd\M\right)^{5/3}\dd t 		\\
&= \frac{2^{2k/3}}{M^{2/3}}\int_0^{\T} \|T_{k-1}\|_{L^{14/5}}^{14/3}\dd t.
\end{align*}
Taking advantage of the interpolation inequality
$$
\|\varphi\|_{L^{14/5}}\leq \|\varphi\|_{L^2}^{4/7}\|\varphi\|^{3/7}_{L^6},
$$
we find that
\begin{align*}
\int_0^{\T}\left(\int_\M |T_k|^{12/5}\dd\M\right)^{5/3}\dd t
&\leq \frac{2^{2k/3}}{M^{2/3}}\int_0^{\T} \|T_{k-1}\|_{L^2}^{8/3}\|T_{k-1}\|^2_{L^6}\dd t\\
&\leq \frac{2^{2k/3}}{M^{2/3}}\sup_{t\in[0,\T]}|T_{k-1}(t)|^{8/3}\int_0^{\T} \|T_{k-1}(t)\|^2\dd t\\
&\leq \frac{2^{2k/3}}{2\kappa_TM^{2/3}} Q_{k-1}^{7/3}.
\end{align*}
Finally, we obtain the nonlinear estimate
\begin{equation}\label{eq:iterest1}
c_1\int_0^{\T}\int_\M |\omega||T_k|^2 \dd\M\dd t\leq \frac{c_1}{\sqrt{2\kappa_T}}\|\vv\|_{L^2(0,\T;H^2)} \frac{2^{k/3}}{M^{1/3}} Q_{k-1}^{7/6}.
\end{equation}
Now for the second term, we use similar arguments as before and obtain
\begin{align*}
\int_0^{\T}\int_\M |\omega||T_k|\dd\M\dd t\leq
\|\vv\|_{L^2(0,\T;H^2)}\left[\int_0^{\T}\left(\int_\M |T_k|^{6/5}\dd\M\right)^{5/3}\dd t \right]^{1/2}.
\end{align*}
Using again \eqref{eq:iter11}, we can estimate the above term as
\begin{align*}
\int_0^{\T}\left(\int_\M |T_k|^{6/5}\dd\M\right)^{5/3}\dd t
&\leq \int_0^{\T}\left(\int_\M |T_{k-1}|^{6/5}\mathds{1}_{\{T_k>0\}}^{8/5}\dd\M\right)^{5/3}\dd t  \\
&\leq \frac{2^{8k/3}}{M^{8/3}}\int_0^{\T}\left(\int_\M |T_{k-1}|^{14/5}\dd\M\right)^{5/3}\dd t 		\\
&= \frac{2^{8k/3}}{M^{8/3}}\int_0^{\T} \|T_{k-1}\|_{L^{14/5}}^{14/3}\dd t.
\end{align*}
Up to the constant in front of the integral, we are in the same situation as above. Hence,
\begin{align*}
\int_0^{\T}\left(\int_\M |T_k|^{6/5}\dd\M\right)^{5/3}\dd t
\leq \frac{2^{8k/3}}{2\kappa_TM^{8/3}} Q_{k-1}^{7/3},
\end{align*}
and thus
\begin{equation}\label{eq:iterest2}
c_1M\int_0^{\T}\int_\M |\omega||T_k| \dd\M\dd t\leq \frac{c_1}{\sqrt{2\kappa_T}}\|\vv\|_{L^2(0,\T;H^2)} \frac{2^{4k/3}}{M^{1/3}} Q_{k-1}^{7/6}.
\end{equation}
The third and fourth terms are estimated in the exact same way as the second term. We therefore obtain
\begin{equation}\label{eq:iterest3}
c_2\int_0^{\T}\int_\M |\omega||T_k| \dd\M\dd t\leq \frac{c_2}{\sqrt{2\kappa_T}}\|\vv\|_{L^2(0,\T;H^2)} \frac{2^{4k/3}}{M^{4/3}} Q_{k-1}^{7/6}.
\end{equation}
and
\begin{equation}\label{eq:iterest4}
2\int_0^{\T}\int_\M |S_T||T_k| \dd\M\dd t\leq \frac{2}{\sqrt{2\kappa_T}}\|S_T\|_{L^2(0,\T;H^1)} \frac{2^{4k/3}}{M^{4/3}} Q_{k-1}^{7/6}.
\end{equation}
In light of \eqref{eq:iter2} and by \eqref{eq:iterest1}--\eqref{eq:iterest4}, the nonlinear iteration may be written as
\begin{equation}\label{eq:iter4}
Q_k\leq \frac{C_0}{M^{1/3}}4^k Q_{k-1}^{7/6} ,\qquad \forall k\geq 1
\end{equation}
where we have assumed $M\geq 1$, and we have set 
$$
C_0:=C_0(\T,\|\vv_0\|,\|T_0\|,\|q_0\|,\|S_T\|_{L^2(0,\T;H^1)},\norm{(\bs S_{\vv}, S_{T}, S_{q})}_{L^2(0,\T; H)})
$$ 
to be the bound
on the velocity field given by Theorem~\ref{thm6.2.1}. Thanks to the nonlinearity in \eqref{eq:iter4}, 
$$
\lim_{k\to\infty} Q_k=0,
$$
provided we choose $M$ sufficiently large. In fact, it is possible to
find that $Q_k$ satisfies the explicit bound
\begin{equation}\label{eq:iter5}
Q_k\leq \left(\frac{Q_0 4^{42}C_0^6}{M^2} \right)^{\left(\frac76\right)^k} \left(\frac{M^{1/3}}{4^7C_0}\right)^6 4^{-6k} ,\qquad \forall k\geq 1.
\end{equation}
Hence, it suffices to impose the condition
$$
\frac{Q_0 4^{42}C_0^6}{M^2}\leq1,
$$ i.e.,
  \begin{equation}\label{eq:argentina}
    M\geq \sqrt{Q_0}2^{42}C_0^3.   
  \end{equation}
In turn, since $Q_0\leq C_0$, we require that
\begin{equation}\label{eq:iter6}
M\geq 2^{42}C_0^{7/2}.
\end{equation}
As $k\to\infty$, we have that $\lambda_k\to M$, and from \eqref{eq:iter0} we learn that
$$
[T-M]^+=\lim_{k\to\infty}T_k =0,
$$
from which the upper bound on the temperature follows.

\appendix

\section{Primitive equations of the atmosphere}\label{sec-kz-idea}
\noindent In \cite{KZ07}, the authors proved the global existence of strong 
solutions for the primitive equations of the ocean with  Dirichlet boundary 
conditions for the side boundary. Adapting these techniques, we are able to prove 
a similar result for the primitive equations of the atmosphere with  free-slip 
boundary conditions for the side boundary. The result in Theorem~\ref{thm6.1.1} below is essential 
for deriving the time-uniform $H^1$-estimate for the velocity $\vv$ and hence the global existence 
of \emph{strong} solutions for the primitive equations with saturation in Section \ref{sec5}.

In this appendix, we decouple the primitive equations of the atmosphere from the temperature and the humidity,
and consider the equation
\begin{equation}\label{eq6.1.1}
\begin{split}
	&\frac{\de \vv}{\de t} + \vv\cdot \nabla\vv + \omega \frac{\de \vv}{\de p}  + \nabla \Phi_s + \A_{\vv} \vv = \boldsymbol S_{\vv},\\
	&\dive\bs v +\frac{\de\omega}{\de p} = 0,
\end{split}
\end{equation}
with initial condition
\begin{equation}
	\bs v(x,y,p,0)=\bs v_0(x,y,p),
\end{equation}
and boundary conditions
\begin{equation}\label{eq6.1.3}
	\left(\nu_{\vv}\frac{\de \vv}{\de p} + \alpha_{\vv}\vv\right)\bigg|_{\Gamma_i}=\omega\big|_{\Gamma_i} = 0,\quad\;  \frac{\de \vv}{\de p}\bigg|_{\Gamma_u} =\omega\big|_{\Gamma_u}= 0,\quad\;  (\vv\cdot \boldsymbol n)\big|_{\Gamma_\ell}=\frac{\de (\vv\cdot\boldsymbol\tau)}{\de \boldsymbol n}\bigg|_{\Gamma_\ell}=0.
\end{equation}
Our result  resembles those in \cite[Theorem~2.1]{KZ07} and \cite[Theorem 2.1]{KZ08}.
The proof of Theorem~\ref{thm6.1.1} below actually follows the lines of \cite[Theorem 2.1]{KZ07} 
and \cite[Theorem 2.1]{KZ08}, where the difference is that instead of the ordinary three-dimensional 
Laplace operator, we consider the more complicated  operator $\A_{\vv}$ 
in \eqref{eq6.1.1}.  
This does not create additional difficulties for the energy estimates below. Nonetheless, 
we present the full details for the proof of Theorem~\ref{thm6.1.1} below 
for the sake of completeness.

\begin{theorem}\label{thm6.1.1}
	Let the space $\V$ be defined as in Section \ref{subsec2.3fs} 
	and assume that $\vv_0\in \V$ and $\bs S_{\vv} \in L^2(0,\T; L^2)$. Then 
	there exists a unique strong solution 
	$$
	\vv\in L^\infty(0,\T; \V)\cap L^2(0,\T; H^2)
	$$ 
	of the primitive equations \eqref{eq6.1.1}--\eqref{eq6.1.3}. 
\end{theorem}

\begin{proof}
	In the following, we will write the norms with respect to the spaces explicitly, that is we write $\norm{\cdot}_{\V}$ denoting the norm on the space $\V$ and $|{\cdot}|$ the absolute value. We also denote by $\norm{\cdot}_{L^p}$ (or $\norm{\cdot}_{L^p(\M)}$) the norm on the space $L^p(\M)$ and by $\norm{\cdot}_{L^p_{x,y}}$ (or $\norm{\cdot}_{L^p(\M')}$) the norm on the space $L^p(\M')$ for $1\leq p\leq \infty$.
	
	Following \cites{KZ07, KZ08}, we first write \eqref{eq6.1.1}$_1$ in component form:
	\begin{equation}\label{eq6.1.4}
		\frac{\de u_k}{\de t} + \sum_{j=1}^3\de_j(u_ju_k)   + \de_k \Phi_s + \A_{\vv} u_k =  S_{k},\quad\quad k=1,2,\\
	\end{equation}
	where 
	$$
	\bs u=(u_1,u_2,u_3)=(\bs v, \omega)=(v_1, v_2, \omega),\quad\quad \bs S_{\vv}=(S_1, S_2).
	$$
	Local existence of strong solutions to \eqref{eq6.1.1}--\eqref{eq6.1.3} is well-known (see e.g.~\cite{PTZ08}) 
	and to prove Theorem~\ref{thm6.1.1} it suffices to show the existence of a constant $\widetilde M>0$ such that $\norm{\bs v(\cdot, t)}_{\V}\leq \widetilde M$ for all $t\in[0,\T]$. 
	
	Let $0\leq\tau_1<\tau_2\leq\T$. In what follows, all the computations are understood for $t\in[\tau_1, \tau_2]$. 
	Before entering into the estimates, we introduce the average operator $M$ in the vertical direction:
	$$
	M\bs v(x,y)=\frac{1}{h}\int_{p_0}^{p_1}\bs v(x,y,p)\dd p,\quad\quad (x,y)\in\M',
	$$
	 where $h:=p_1-p_0$.
	
	We first obtain an $L^6$-estimate on $\bs v$. We multiply \eqref{eq6.1.4} by $u_k^5$, where $k=1,2$, integrate over $\M$, and sum the resulting equations. We arrive at
	\begin{equation}\label{eq6.1.5}
	\begin{split}
		\frac16\sum_{k=1}^2\ddt\norm{u_k}_{L^6}^6 + \frac{5\mu_{\vv}}{9} \sum_{k=1}^2\int_{\M}\abs{\nabla(u_k^3)}^2\dd\M + \frac{5\nu_{\vv}}{9} \sum_{k=1}^2\int_{\M} \left(\frac{gp}{R\overline T}\right)^2\abs{\de_p(u_k)}^3\dd\M \\
		+\alpha_{\vv}\sum_{k=1}^2\int_{\Gamma_i} \left(\frac{gp_1}{R\overline T}\right)^2\abs{u_k}^6\,\dd\Gamma_i
		=-\sum_{k=1}^2\int_{\M} u_k^5 \de_k\Phi_s\dd\M + \sum_{k=1}^2\int_{\M} S_ku_k^5\dd\M\\
		=-h\sum_{k=1}^2\int_{\M'} M(u_k^5) \de_k\Phi_s\dd\M' + \sum_{k=1}^2\int_{\M} S_ku_k^5\dd\M.\\
	\end{split}
	\end{equation}
	As in \cite{KZ07}, the first term on the right-hand side of \eqref{eq6.1.5} can be bounded  using H\"older's inequality and the Sobolev embedding $W^{1,6/5}(\M')\hookrightarrow L^3(\M')$ by
	\begin{equation*}
		C \sum_{k=1}^2\norm{ M(u_k^5) }_{L^3} \norm{\nabla\Phi_s}_{L^{3/2}}
			\leq C\sum_{k=1}^2( \norm{ \nabla M(u_k^5) }_{L^{6/5}}  + \norm{ M(u_k^5) }_{L^{6/5}} )\norm{\nabla\Phi_s}_{L^{3/2}_{x,y}}.
	\end{equation*}
	Now, we estimate
	$$
	\norm{ \nabla M(u_k^5) }_{L^{6/5}} \leq C\norm{u_k^2\nabla(u_k^3)}_{L^{6/5}} \leq CJ^2\bar J^3,
	\quad
	\norm{ M(u_k^5) }_{L^{6/5}}\leq C\norm{u_k}_{L^6}^5\leq C J^5\leq CJ^2\bar J^3,
	$$
	where we denoted 
	$$
	J(t)=\big(\norm{(u_1^3, u_2^3)}_{L^2}^2\big)^{1/6} =\big(\sum_{k=1}^2\norm{u_k}_{L^6}^6\big)^{1/6},
	$$
	and
	\begin{equation*}\begin{split}
			\bar J(t) =& \big(\norm{(u_1^3, u_2^3)}_{\V}^2\big)^{1/6} =\bigg(\sum_{k=1}^2\int_{\M}\abs{\nabla(u_k^3)}^2\dd\M  \\
			 &+\sum_{k=1}^2\int_{\M} \left(\frac{gp}{R\overline T}\right)^2\abs{\de_p(u_k^3)}^2\dd\M 
		+\sum_{k=1}^2\int_{\Gamma_i} \left(\frac{gp_1}{R\overline T}\right)^2\abs{u_k}^6\,\dd\Gamma_i \bigg)^{1/6},
	\end{split}\end{equation*}
	and used the Poincar\'e's inequality, 
	\begin{equation}\label{eq6.1.7}
		J(t)\leq C \bar J(t).
	\end{equation}
	The second term on the right side of \eqref{eq6.1.5} can be estimated by
	\begin{equation*}\begin{split}
\sum_{k=1}^2\norm{S_k}_{L^2}\norm{u_k^5}_{L^2}=\sum_{k=1}^2\norm{S_k}_{L^2}\norm{u_k^3}_{L^{10/3}}^{5/3}\leq \sum_{k=1}^2\norm{S_k}_{L^2}\norm{u_k^3}_{L^2}^{2/3}\norm{u_k^3}_{L^2},
	\end{split}\end{equation*}
	which is further bounded using the Poincar\'e inequality by
	\begin{equation*}
		\sum_{k=1}^2\norm{S_k}_{L^2}\norm{u_k}_{L^6}^{2}\norm{(u_1^3, u_2^3)}_{\V}\leq C \mathbf SJ^2\bar J^3,
	\end{equation*}
	where
	$$
	\mathbf S(t)=\left(\sum_{k=1}^2\norm{S_k(t)}_{L^2}^2\right)^{1/2}.
	$$
	Now, we recall that $\kappa_{\vv}=\min(\mu_{\vv}, \nu_{\vv}, \alpha_{\vv})$ and  deduce from \eqref{eq6.1.5} that
	$$
	\frac16 \ddt J^6 + \frac59 \kappa_{\vv}\bar J^6 \leq C\norm{\nabla\Phi_s}_{L^{3/2}_{x,y}}J^2\bar J^3 + C\mathbf SJ^2\bar J^3,
	$$
	which, by Young's inequality, implies that
	\begin{equation}\label{eq6.1.8}
		\ddt J^6 + \kappa_{\vv}\bar J^6 \leq C\norm{\nabla\Phi_s}^2_{L^{3/2}_{x,y}}J^4 + C\mathbf S^2 J^4.
	\end{equation}
We now set
	$$
	K(t)=\big( \sum_{k=1}^2\norm{\de_pu_k(\cdot, t)}_{L^2}^2\big)^{1/2},\quad\quad
	$$
	and
	\begin{equation*}\begin{split}
		\bar K(t) = \bigg(\sum_{k=1}^2\int_{\M}\abs{\nabla(\de_pu_k)}^2\dd\M + \sum_{k=1}^2\int_{\M} \left(\frac{gp}{R\overline T}\right)^2\abs{\de_{pp}(u_k)}^2\dd\M \\
	+\frac{\mu_{\vv}}{\nu_{\vv}}\sum_{k=1}^2\int_{\Gamma_i} \abs{\nabla u_k}^2\dd\Gamma_i\bigg)^{1/2}.
	\end{split}\end{equation*}
	In order to obtain the estimates for $K$ and $\bar K$, we multiply \eqref{eq6.1.4} by $-\de_{pp}u_k$, where $k=1,2$ integrate over $\M$, and sum the resulting equations together; we find
	\begin{equation}\begin{split}\label{eq6.1.10}
		\frac12\ddt K^2 + \kappa_{\vv}\bar K^2 = \sum_{j,k=1}^2\int_{\M}u_j\de_ju_k\de_{pp}u_k\dd\M
		+\sum_{k=1}^2\int_{\M} u_3\de_pu_k\de_{pp}u_k\dd\M\\
		+\sum_{k=1}^2\int_{\M} \de_k\Phi_s\de_{pp}u_k\dd\M
		-\sum_{k=1}^2\int_{\M} S_k\de_{pp}u_k \dd\M.
	\end{split}\end{equation}
	Integrations by parts on the first two terms on the right-hand side of \eqref{eq6.1.10} yield
	\begin{equation*}\begin{split}
		 \sum_{j,k=1}^2\int_{\M}u_j\de_ju_k\de_{pp}u_k\dd\M
		 =-\sum_{j,k=1}^2\int_{\M} \de_pu_j\de_ju_k\de_pu_k\dd\M -\frac{\alpha_{\vv}}{\nu_{\vv}} \sum_{j,k=1}^2\int_{\Gamma_i} u_j\de_ju_k u_k \dd\Gamma_i \\+
		 \frac12\sum_{j,k=1}^2\int_{\M} \de_ju_j\de_pu_k\de_pu_k\dd\M
		 =\sum_{j,k=1}^2\int_{\M}\de_{pj}u_j u_k \de_pu_k \dd\M
		 + \sum_{j,k=1}^2\int_{\M}\de_{p}u_j u_k \de_{pj}u_k \dd\M\\
		 +\frac{\alpha_{\vv}}{2\nu_{\vv}} \sum_{j,k=1}^2\int_{\Gamma_i} \de_ju_ju_ku_k \dd\Gamma_i 
		 -\sum_{j,k=1}^2\int_{\M} u_j\de_{pj}u_k\de_pu_k\dd\M,\\
	\end{split}\end{equation*}
	and with the divergence free condition on $\bs u$, we have
	\begin{equation*}\begin{split}
		\sum_{k=1}^2\int_{\M} u_3\de_pu_k\de_{pp}u_k\dd\M
		= -\frac12 \sum_{k=1}^2\int_{\M} \de_pu_3\de_pu_k\de_{p}u_k\dd\M\\
		= \frac12\sum_{j,k=1}^2\int_{\M} \de_ju_j\de_pu_k\de_pu_k\dd\M=-\sum_{j,k=1}^2\int_{\M} u_j\de_{pj}u_k\de_pu_k\dd\M.
	\end{split}\end{equation*}
	The right-hand side of \eqref{eq6.1.10} is now less than or equal to
	\begin{equation}\begin{split}\label{eq6.1.11}
		C\sum_{j,k=1}^2\norm{u_k}_{L^6}\norm{\de_pu_k}_{L^3}\norm{\nabla(\de_pu_j)}_{L^2}
		+C\sum_{j,k=1}^2\norm{u_k}_{L^6}\norm{\de_pu_j}_{L^3}\norm{\nabla(\de_pu_k)}_{L^2}\\
		+C\sum_{j,k=1}^2\norm{u_j}_{L^6}\norm{\de_pu_k}_{L^3}\norm{\nabla(\de_pu_k)}_{L^2}
		+\sum_{k=1}^2\norm{S_k}_{L^2}\norm{\de_{pp}u_k}_{L^2}\\
		+\sum_{k=1}^2\left|\int_{\M} \de_k\Phi_s\de_{pp}u_k \dd\M \right|
		+\frac{\alpha_{\vv}}{2\nu_{\vv}} \sum_{j,k=1}^2\left|\int_{\Gamma_i} \de_ju_ju_ku_k  \dd\Gamma_i \right|.
	\end{split}\end{equation}
	The first three terms in \eqref{eq6.1.11} are bounded by $CJK^{1/2}\bar K^{3/2}$ and the fourth term is bounded by $C\mathbf S\bar K$. We rewrite the fifth term in \eqref{eq6.1.11} as
	\begin{equation}\begin{split}
		\sum_{k=1}^2\left|\int_{\M'} \de_k\Phi_s\int_{p_0}^{p_1}\de_{pp}u_k \dd p\,\dd\M' \right|
		=\sum_{k=1}^2\left|\int_{\M'} \de_k\Phi_s (\de_{p}u_k)\big|_{p=p_0}^{p=p_1} \dd\M' \right|\\
		=\frac{\alpha_{\vv}}{\nu_{\vv}}\sum_{k=1}^2\left|\int_{\M'} \de_k\Phi_s u_k\big|_{p=p_1} \dd\M' \right|
		\leq \frac{\alpha_{\vv}}{\nu_{\vv}}\sum_{k=1}^2 \norm{\de_k\Phi_s}_{L^{3/2}_{x,y}} \norm{ u_k\big|_{p=p_1} }_{L^3_{x,y}},
	\end{split}\end{equation}
	and we have by the trace theorem
	$$
	\norm{ u_k\big|_{p=p_1} }_{L^3_{x,y}} \leq C\norm{ u_k\big|_{p=p_1} }_{L_{x,y}^4}
	\leq C\norm{ u_k\big|_{p=p_1} }_{H_{x,y}^{1/2}}
	\leq C\norm{ u_k }_{H^1(\M)} \leq C\bar E.
	$$
	where
	$$
	\bar E(t) =\norm{\bs v(\cdot, t) }_{H^1(\M)} = \bigg(\sum_{k=1}^2 \norm{ u_k(\cdot, t) }_{H^1(\M)}^2 \bigg)^{1/2} .
	$$
	By H\"older's and Young's inequalities, the last term in \eqref{eq6.1.11} is estimated by
	$$
		C\sum_{j,k=1}^2\norm{\nabla u_j}_{L^2(\Gamma_i)}\norm{u_k^2}_{L^3(\Gamma_i)}
		\leq \frac{\kappa_{\vv}}8\bar K^2 + \frac{\kappa_{\vv}}8 \bar J^6 + C.
	$$
	Now, we can deduce from \eqref{eq6.1.10} that
	\begin{equation*}\begin{split}
		\frac12\ddt K^2 + \kappa_{\vv}\bar K^2 \leq CJK^{1/2}\bar K^{3/2} + C\norm{\nabla\Phi_s}_{L^{3/2}_{x,y}}\bar E + C\mathbf S\bar K + \frac{\kappa_{\vv}}8\bar K^2 + \frac{\kappa_{\vv}}8 \bar J^6 + C,
	\end{split}\end{equation*}
	whence by noticing that $K(t)\leq \bar E(t)$,
	\begin{equation}\label{eq6.1.13}
		\ddt K^2 + \kappa_{\vv}\bar K^2 \leq C\bar E^2J^4 + C\norm{\nabla\Phi_s}_{L^{3/2}_{x,y}}^2 + \bar E^2 + C\mathbf S^2 +  \frac{\kappa_{\vv}}8 \bar J^6 + C.
	\end{equation}
	
	Next, we need an estimate on $\norm{\nabla\Phi_s}_{L^2_tL^{3/2}_{x,y}}=\norm{\nabla\Phi_s}_{L^2_tL^{3/2}_{x,y}([\tau_1,\tau_2]\times\M')}$. For this purpose, we average the primitive equations in the vertical direction to obtain
	\begin{equation}\begin{split}
		&\frac{\de Mu_k }{\de t} - \mu_{\vv}\Delta Mu_k + \de_k \Phi_s = -\frac{\alpha_{\vv}}{p_1-p_0}\bigg(\frac{gp_1}{R\overline T}\bigg)^2 u_k\big|_{p=p_1},\\
		&\hspace{150pt}-\sum_{k=1}^2M(\de_j(u_ju_k)) + MS_k,\quad\quad k=1,2,\\
		&\de_1Mu_1 + \de_2Mu_2 = 0.
	\end{split}\end{equation}
	This is a linear two-dimensional Stokes problem for $(M\bs v, \Phi_s)$ with initial datum
	$$
	M\bs v(\cdot, t)\big|_{t=\tau_1} = M\bs v(\cdot, \tau_1).
	$$
	By the $L_t^qL_{x,y}^p$ estimates for the Stokes problem due to Sohr and von Wahl \cite{SVW86}, we have
	\begin{equation*}\begin{split}
		\norm{\nabla\Phi_s}_{L^2_tL^{3/2}_{x,y}} \leq C\norm{ u_k\big|_{p=p_1} }_{L^2_tL^{3/2}_{x,y}} + C\sum_{j,k=1}^2\norm{ M(\de_j(u_ju_k))}_{L^2_tL^{3/2}_{x,y}}\\
		+C\sum_{k=1}^2\norm{MS_k }_{L^2_tL^{3/2}_{x,y}}
		+C\sum_{k=1}^2\norm{ \nabla u_k(\cdot,\tau_1)}_{L^{2}}\\
		\leq C\norm{ u_k\big|_{p=p_1} }_{L^2_tL^{3/2}_{x,y}}
		+C\sum_{j,k=1}^2\norm{ u_j\de_ju_k }_{L^2_tL^{3/2}(\M)}\\
		+C\sum_{k=1}^2\norm{S_k }_{L^2_tL^{2}(\M)}
		+C\norm{ \nabla\bs v(\cdot,\tau_1)}_{L^2(\M)}
		=:I_1+ I_2 + I_3 +I_4.
	\end{split}\end{equation*}
	As before, by the trace theorem, the first term $I_1$ is estimated as
	$$
	I_1 \leq C\norm{ \bs v\big|_{p=p_1} }_{L^2_tL^{4}_{x,y}}
	\leq C\norm{ \norm{\bs v}_{H^1(\M)} }_{L^2_t}
	\leq C\norm{ \bar E }_{L^2_t}
	$$
	The second term $I_2$ is estimated as
	$$
	I_2 \leq C \sum_{j,k=1}^2\norm{ \norm{u_j}_{L^6}\norm{\de_ju_k}_{L^2} }_{L^2_t}
	\leq C\norm{ J\bar E }_{L^2_t}.
	$$
	Collecting the estimates for $I_1$ and $I_2$, we obtain
	\begin{equation}\label{eq6.1.15}
		\norm{\nabla\Phi_s}_{L^2_tL^{3/2}_{x,y}}^2
		\leq C\norm{ \bar E }_{L^2_t}^2 + C\norm{ J\bar E }_{L^2_t}^2 + C\norm{\mathbf S}_{L^2_t}^2 + C\norm{ \nabla \bs v(\cdot,\tau_1)}_{L^2}^2.
	\end{equation}
	Now, considering $0\leq\tau_1<\tau_2<\tau_3$ and integrating \eqref{eq6.1.8} on $[\tau_1, \tau_2]$ yields
	$$
	J(\tau_2)^6 + \kappa_{\vv}\norm{ \bar J^3 }_{L_t^2(\tau_1, \tau_2)}^2
	\leq J(\tau_1)^6 + \norm{\nabla\Phi_s}_{L^2_tL^{3/2}_{x,y}}^2\sup_{\tau_1\leq\tau\leq\tau_3}J(\tau)^4
	+C\norm{\mathbf S}_{L_t^2}^2\sup_{\tau_1\leq\tau\leq\tau_3}J(\tau)^4,
	$$
	where the $L_t^2$-norms are taken over $[\tau_1, \tau_3]$. Implementing the pressure estimate \eqref{eq6.1.15}, we obtain
	\begin{equation*}\begin{split}
		J(\tau_2)^6 + \kappa_{\vv}\norm{ \bar J^3 }_{L_t^2(\tau_1, \tau_2)}^2
			\leq J(\tau_1)^6 + C\norm{ \bar E }_{L^2_t}^2\sup_{\tau_1\leq\tau\leq\tau_3}J(\tau)^4
			+C\norm{ \bar E }_{L^2_t}^2\sup_{\tau_1\leq\tau\leq\tau_3}J(\tau)^6\\
			+C\big( \norm{\mathbf S}_{L_t^2}^2 + \norm{ \nabla \bs v(\cdot,\tau_1)}_{L^{2}}^2 \big)\sup_{\tau_1\leq\tau\leq\tau_3}J(\tau)^4,
	\end{split}\end{equation*}
	which is valid for all $\tau_2\in[\tau_1, \tau_3]$. Hence, taking $\sup_{\tau_1\leq\tau\leq\tau_3}$ on both sides and using 
	the Cauchy-Schwarz inequality to absorb the terms involving $\sup_{\tau_1\leq\tau\leq\tau_3}J(\tau)^4$ into the left-hand 
	side, we obtain
	\begin{equation}\begin{split}\label{eq6.1.16}
		\sup_{\tau_1\leq \tau\leq\tau_3}J(\tau)^6+ \kappa_{\vv}\norm{ \bar J^3 }_{L_t^2(\tau_1, \tau_3)}^2
			\leq 2J(\tau_1)^6 + C\norm{ \bar E }_{L^2_t}^4
			+C\norm{ \bar E }_{L^2_t}^2\sup_{\tau_1\leq\tau\leq\tau_3}J(\tau)^6\\
			+C\big( \norm{\mathbf S}_{L_t^2}^4 + \norm{ \nabla \bs v(\cdot,\tau_1)}_{L^{2}}^4 \big).
	\end{split}\end{equation}
	Similarly, we integrate \eqref{eq6.1.13} on $[\tau_1, \tau_2]$ and use the pressure estimate \eqref{eq6.1.15} to find
	\begin{equation*}\begin{split}
		K(\tau_2)^2 &+ \kappa_{\vv}\norm{\bar K}_{L_t^2(\tau_1, \tau_2)}^2
		\leq K(\tau_1)^2 + C\norm{\bar E}_{L_t^2}^2 \sup_{\tau_1\leq\tau\leq\tau_3}J(\tau)^4
		+ C\norm{ \bar E }_{L^2_t}^2 \\
		&+ C\norm{ \bar E }_{L^2_t}^2\sup_{\tau_1\leq\tau\leq\tau_3}J(\tau)^2+\frac{\kappa_{\vv}}8 \norm{ \bar J^3 }_{L_t^2}^2
		+C\big( \norm{\mathbf S}_{L_t^2}^2 + \norm{ \nabla \bs v(\cdot,\tau_1)}_{L^{2}}^2 +1 \big).
	\end{split}\end{equation*}
	Therefore, 
	\begin{equation}\begin{split}\label{eq6.1.17}
		\sup_{\tau_1\leq \tau\leq\tau_3}&K(\tau)^2 + \kappa_{\vv}\norm{\bar K}_{L_t^2(\tau_1, \tau_3)}^2
		\leq K(\tau_1)^2 + C\norm{\bar E}_{L_t^2}^2 \sup_{\tau_1\leq\tau\leq\tau_3}J(\tau)^4
		+ C\norm{ \bar E }_{L^2_t}^2 \\
		&+ C\norm{ \bar E }_{L^2_t}^2\sup_{\tau_1\leq\tau\leq\tau_3}J(\tau)^2+\frac{\kappa_{\vv}}8 \norm{ \bar J^3 }_{L_t^2}^2
		+C\big( \norm{\mathbf S}_{L_t^2}^2 + \norm{ \nabla \bs v(\cdot,\tau_1)}_{L^{2}}^2 +1 \big).
	\end{split}\end{equation}	
	Now, choose $\delta>0$ such that
	\begin{equation}\label{eq6.1.18}
		\norm{\bar E}_{L_t^2(t, t+2\delta)}^2\leq \frac1\gamma,\quad\quad\forall\,t \in[0, \T],
	\end{equation}
	where $\gamma$ is a sufficiently large constant to be determined later. Let $t^0=0$ and then for $j=1,\cdots,l$, choose $t^j\in(j\delta, (j+1)\delta)$ such that
	\begin{equation}\label{eq6.1.19}
		\norm{\bs v(\cdot, t^j)}_{H^1(\M)}^2 \leq \frac1\delta\int_{j\delta}^{(j+1)\delta} \norm{\bs v(\cdot, \tau)}_{H^1(\M)}^2\dd\tau \leq \frac{1}{\delta \gamma},
	\end{equation}
	where $l$ is the largest integer such that $(l+1)\delta\leq \T$. Let $t^{l+1}=\T$. Note that 
	\begin{equation}
		t^{j+1}-t^j\leq 2\delta,\quad\quad\forall j=0,\ldots,l.
	\end{equation}
	Summing the inequalities \eqref{eq6.1.16} and \eqref{eq6.1.17} with $\tau_1=t^j$ and $\tau_3=t^{j+1}$ implies that
	\begin{equation*}\begin{split}
		\sup_{t^j\leq t\leq t^{j+1}}J(t)^6+ \frac{7\kappa_{\vv}}8\norm{ \bar J^3 }_{L_t^2}^2
		+\sup_{t^j\leq t\leq t^{j+1}}K(t)^2 + \kappa_{\vv}\norm{\bar K}_{L_t^2}^2
		\leq 2J(t^j)^6 + K(t^j)^2
		+C\norm{ \bar E }_{L^2_t}^4\\
			+C\norm{ \bar E }_{L^2_t}^2\sup_{\tau_1\leq\tau\leq\tau_3}J(\tau)^6
			+C\norm{\bar E}_{L_t^2}^2 \sup_{\tau_1\leq\tau\leq\tau_3}J(\tau)^4
		+ C\norm{ \bar E }_{L^2_t}^2 
		+ C\norm{ \bar E }_{L^2_t}^2\sup_{\tau_1\leq\tau\leq\tau_3}J(\tau)^2\\
		+C\big( \norm{\mathbf S}_{L_t^2}^4 + \norm{ \nabla \bs v(\cdot,t^j)}_{L^2}^4 \big)+C\big( \norm{\mathbf S}_{L_t^2}^2 + \norm{ \nabla \bs v(\cdot,t^j)}_{L^{2}}^2 +1 \big),
	\end{split}\end{equation*}
	which, by using Young's inequality and absorbing $\sup_{\tau_1\leq\tau\leq\tau_3}J(\tau)^4$ and $\sup_{\tau_1\leq\tau\leq\tau_3}J(\tau)^2$ into the left-hand side, yields
	\begin{equation*}\begin{split}
		\sup_{t^j\leq t\leq t^{j+1}}J(t)^6
		+\sup_{t^j\leq t\leq t^{j+1}}K(t)^2 + \kappa_{\vv}\norm{\bar K}_{L_t^2}^2
		\leq 4J(t^j)^6 + 2K(t^j)^2
		+C\norm{ \bar E }_{L^2_t}^4\\
			+C\norm{ \bar E }_{L^2_t}^2\sup_{\tau_1\leq\tau\leq\tau_3}J(\tau)^6
			+C\big( \norm{\mathbf S}_{L_t^2}^4 + \norm{ \nabla \bs v(\cdot,t^j)}_{L^{2}}^4 +1\big).
	\end{split}\end{equation*}
	Using \eqref{eq6.1.18} and \eqref{eq6.1.19} and 
	$$
	\max\left\{J(t^j)^2,K(t^j)^2\right\} \leq C\bar E(t^j)^2=C\norm{ \bs v(\cdot, t^j) }_{H^1(\M)}^2 \leq \frac{C}{\delta\gamma},
	$$
	we obtain
	\begin{equation}\begin{split}\label{eq6.1.20}
		\sup_{t^j\leq t\leq t^{j+1}}J(t)^6
		+\sup_{t^j\leq t\leq t^{j+1}}K(t)^2 + \kappa_{\vv}\norm{\bar K}_{L_t^2}^2
		\leq \frac{C}{\delta^3\gamma^3} + \frac{C}{\delta\gamma} \\
		+ \frac{C}{\gamma^2}+\frac{C}{\gamma}\sup_{\tau_1\leq\tau\leq\tau_3}J(\tau)^6
		+C\big( \norm{\mathbf S}_{L_t^2}^4 + \frac{C}{\delta^2\gamma^2} +1\big).
	\end{split}\end{equation}
	Therefore, choosing $\gamma>0$ large enough to absorb the fourth term in in the right-hand side of \eqref{eq6.1.20}, we eventually find that
	\begin{equation}\label{eq6.1.21}
		J(t),\, K(t),\, \norm{\bar K}_{L^2(0,\T)} \leq \widetilde M,\quad\quad t\in[0,\T],
	\end{equation}
	where 
	$$
	\widetilde M:=\biggl(\frac{C}{\delta^3\gamma^3} + \frac{C}{\delta\gamma} + \frac{C}{\gamma^2}
		+C\big( \norm{\mathbf S}_{L_t^2}^4 + \frac{C}{\delta^2\gamma^2} +1\big)\biggr)\left(\frac1{\kappa_{\vv}} + 1\right).
	$$
	We are now in position to bound $\norm{\bs v(\cdot,t)}_{\V}$ for all $t\in[0,\T]$ with a constant independent of $\T$. Taking the inner product of each
	side of \eqref{eq6.1.1} with $A_{\vv}\vv$ leads to
	\begin{equation*}\begin{split}
		\frac12\ddt\norm{A_{\vv}^{1/2} \vv }_{L^2}^2 + \norm{ A_{\vv}\vv}_{L^2}^2
		&\leq \sum_{k=1}^2\norm{u_j\de_j\vv}_{L^2}\norm{ A_{\vv}\vv}_{L^2}+\norm{u_3\de_{p}\bs v}_{L^2}\norm{ A_{\vv}\vv}_{L^2}\\ 
		&\quad+ \norm{S_{\vv}}_{L^2}\norm{ A_{\vv}\vv}_{L^2},
	\end{split}\end{equation*}
	which implies that
	\begin{equation}\begin{split}
		\ddt\norm{A_{\vv}^{1/2} \vv }_{L^2}^2 &+ \norm{ A_{\vv}\vv}_{L^2}^2
		\leq \sum_{k=1}^2\norm{u_j\de_j\vv}_{L^2}^2+\norm{u_3\de_{p}\bs v}_{L^2}^2 + C\norm{S_{\vv}}_{L^2}^2\\
		&\leq CJ^2\norm{A_{\vv}^{1/2} \vv }_{L^2}\norm{ A_{\vv}\vv}_{L^2} + CK\bar K\norm{A_{\vv}^{1/2} \vv }_{L^2}\norm{ A_{\vv}\vv}_{L^2} + C\norm{S_{\vv}}_{L^2}^2,
	\end{split}\end{equation}
	where we used the anisotropic estimates for the second term in the right-hand side. Therefore,
	$$
	\ddt\norm{A_{\vv}^{1/2} \vv }_{L^2}^2 + \norm{ A_{\vv}\vv}_{L^2}^2
		\leq CJ^4\norm{A_{\vv}^{1/2} \vv }_{L^2}^2 + CK^2\bar K^2\norm{A_{\vv}^{1/2} \vv }^2_{L^2} + C\norm{S_{\vv}}_{L^2}^2,
	$$
and the uniform bound  of $\norm{\bs v(t)}_{\V}$ follows from the Gronwall lemma and the estimate \eqref{eq6.1.21}. Hence, we proved the global existence of Theorem~\ref{thm6.1.1}. The uniqueness follows similarly as in \cite[pp.~2748]{KZ07} and since it is not important for our goal here, we thus omit the details. This ends the proof of Theorem~\ref{thm6.1.1}.
\end{proof}

\section*{Acknowledgments}
\noindent MCZ thanks Professor C.G. Gal for suggesting the use of the De Giorgi technique for the maximum
principle results in Section \ref{sec:maxprinc}.
MCZ, AH, and RT were partially supported by the National Science Foundation 
under the grant NSF DMS-1206438 and by the Research Fund of Indiana University,
IK was supported in part by the NSF grant DMS-1311943, while MZ was supported
in part by the NSF grant DMS-1109562.

\end{document}